\documentclass[]{article}
\usepackage{graphicx}
\usepackage{amsfonts}
\usepackage{amsmath}
\usepackage{amssymb}
\usepackage{fancyhdr}
\usepackage{titlesec}
\usepackage{indentfirst}
\usepackage{booktabs}
\usepackage{verbatim}
\usepackage{color}
\usepackage{amsthm}
\usepackage{subfigure}

\usepackage[page,header]{appendix}
\usepackage{titletoc}

\newcommand{\rd}{\,\mathrm{d}}

\numberwithin{equation}{section}
\newtheorem{theorem}{Theorem}[section]
\newtheorem{lemma}[theorem]{Lemma}
\newtheorem{corollary}[theorem]{Corollary}

\newtheorem{remark}[theorem]{Remark}

\def\tu{\tilde{u}}
\def\tU{\tilde{U}}
\def\eps{\varepsilon}

\begin{document}

\title{On the uniform accuracy of implicit-explicit backward differentiation formulas (IMEX-BDF) for stiff hyperbolic relaxation systems and kinetic equations\footnote{JH's research was supported in part by NSF grant DMS-1620250 and NSF CAREER grant DMS-1654152. RS's research was supported in part by NSF grants DMS-1613911, RNMS-1107444 (KI-Net) and ONR grant N00014-1812465.}}

\author{Jingwei Hu\footnote{Department of Mathematics, Purdue University, West Lafayette, IN 47907, USA (jingweihu@purdue.edu).} \  \  
	    and \ Ruiwen Shu\footnote{Department of Mathematics, University of Maryland, College Park, MD 20742, USA (rshu@cscamm.umd.edu).}}    
\maketitle

\begin{abstract}
Many hyperbolic and kinetic equations contain a non-stiff convection/transport part and a stiff relaxation/collision part (characterized by the relaxation or mean free time $\varepsilon$). To solve this type of problems, implicit-explicit (IMEX) multistep methods have been widely used and their performance is understood well in the non-stiff regime ($\varepsilon=O(1)$) and limiting regime ($\varepsilon\rightarrow 0$). However, in the intermediate regime (say, $\varepsilon=O(\Delta t)$), uniform accuracy has been reported numerically without a complete theoretical justification (except some asymptotic or stability analysis). In this work, we prove the uniform accuracy -- an optimal {\it a priori} error bound -- of a class of IMEX multistep methods, IMEX backward differentiation formulas (IMEX-BDF), for linear hyperbolic systems with stiff relaxation. The proof is based on the energy estimate with a new multiplier technique. For nonlinear hyperbolic and kinetic equations, we numerically verify the same property using a series of examples.
\end{abstract}

{\small 

}

\section{Introduction}

Many hyperbolic and kinetic equations contain a non-stiff convection/transport part and a stiff relaxation/collision part. For example, a simple linear hyperbolic system with stiff relaxation (cf.~\cite{CLL94}) reads:
\begin{equation}\label{linearhyp}\left\{\begin{split}
& \partial_t u + \partial_x v = 0, \\
& \partial_t v +  \partial_x u = \frac{1}{\eps}(bu-v),
\end{split}\right.\end{equation}
where $u=u(t,x)$, $v=v(t,x)$ are the unknown functions of time $t\geq 0$ and position $x \in\mathbb{R}$ or $\mathbb{T}=[0,2\pi]$ (a torus), $b$ is a constant such that $|b|<1$, and $\eps>0$ is the relaxation parameter. Depending on the application, $\eps$ could take any value between 0 and $O(1)$, leading to non-stiff regime ($\eps=O(1)$), stiff regime ($\eps\ll 1$), or intermediate regime ($0<\eps<1$, neither too small nor too big). In particular, when $\eps\rightarrow 0$, the second equation of (\ref{linearhyp}) formally implies $v\rightarrow bu$. Substituting it into the first equation yields 
\begin{equation} \label{con}
\partial_t u + \partial_x (bu)=0,
\end{equation}
which is the so-called zero relaxation limit. When $0<\eps \ll 1$, via the Chapman-Enskog expansion, one can derive the next order approximation:
\begin{equation} \label{condiff}
\partial_t u +\partial_x (bu)=\eps (1-b^2)\partial_{xx}u,
\end{equation}
which is a convection-diffusion equation.

Due to the multiscale nature of the problem (\ref{linearhyp}) (or hyperbolic/kinetic equations of a similar kind), a popular numerical method for time discretization is the so-called implicit-explicit (IMEX) schemes, being either multistage (Runge-Kutta type) \cite{Jin95, JX95, ARS97, PR05, FJ10, DP13} or multistep \cite{Crouzerix80, ARW95, HR07, DP17, ADP19}. In these schemes, the non-stiff convection part is treated explicitly, while the stiff relaxation part is treated implicitly. In this way, the schemes are expected to be stable for $\varepsilon$ ranging between $0$ and $O(1)$, provided the time step $\Delta t$ only satisfies the CFL condition from the convection part. Furthermore, the high order accuracy of the schemes can often be guaranteed when $\eps=O(1)$ (being the order of the full IMEX schemes) and when $\eps \rightarrow 0$ (being the order of the explicit part of the IMEX schemes). The latter is a direct consequence of the schemes being asymptotic-preserving (AP) \cite{Jin99}. In the intermediate regime (say, $\eps =O(\Delta t)$), many IMEX-RK schemes would suffer from order reduction \cite{Boscarino07} and special schemes can be designed to alleviate this by imposing additional order conditions \cite{Boscarino09, BR09, BP17}. On the other hand, the IMEX multistep schemes can often maintain the uniform accuracy for a wide range of  $\varepsilon$ as suggested by strong numerical evidence \cite{HR07, DP17, ADP19}.


Motivated by the above observation, we study in this paper the stability and accuracy of a class of IMEX multistep schemes, IMEX backward differentiation formulas (IMEX-BDF), for (\ref{linearhyp}) and its generalization to the variable coefficient case. We rigorously prove the {\it uniform stability and accuracy} of the IMEX-BDF schemes up to fourth order when coupled with spectral discretization in space. Simply speaking, our main result implies the following:
\begin{equation} \label{error}
\|u(T,\cdot)-U(T,\cdot)\|_{L^2}+\|v(T,\cdot)-V(T,\cdot)\|_{L^2} \leq C \Delta t^q,
\end{equation}
where $U$ and $V$ are the numerical solutions at time $T$, $q$ is the order of the scheme, and $C$ is a constant depending on $T$, $q$, etc., but {\it independent of $\eps$}. Note that (\ref{error}) is an optimal {\it a priori} error bound that holds regardless of the value of $\eps$, hence the accuracy of the scheme is guaranteed in all regimes! This is especially useful when $\eps$ is neither too small nor too big. For instance, if we are interested in capturing the behavior of the solution at $O(\eps)$, or in other words, the solution to the limiting equation (\ref{condiff}), all we need is to resolve the $O(\eps)$ term in the sense that $\Delta t^q< O(\eps)$. Here the advantage of using high order (at least second order) schemes should be clear: first order schemes are not able to capture $O(\eps)$ information since $\Delta t$ is generally bigger than $O(\eps)$ when $\eps$ is small. 

Our proof is based on the energy estimate along with a new multiplier technique. Although the multiplier technique for studying the stability of multistep methods for ODEs traces back to \cite{Dahlquist78,NO81}, it has only recently been used in the numerical analysis of parabolic problems, see for instance \cite{LMV13, AL15, AK16}. Existing multipliers cannot be applied to study our problem (\ref{linearhyp}) due to its special convection-relaxation structure. Therefore, a new class of multipliers needs to be invented to prove the stability which constitutes one of our main contributions of this work. The uniform stability, combined with the consistency of the IMEX-BDF schemes which is a consequence of the uniform-in-$\eps$ regularity of \eqref{linearhyp}, finally yields the desired uniform accuracy. To the best of our knowledge, this is the first rigorous high order uniform-in-$\eps$ accuracy result established for stiff hyperbolic type equations.

Although we used the linear system (\ref{linearhyp}) as a prototype problem to prove the main result, we expect similar conclusion holds for more general nonlinear hyperbolic equations, among which a very important class of equations in multiscale modeling is the following kinetic equation (cf.~\cite{Villani02}):
\begin{equation} \label{kinetic}
\partial_t f+ v\cdot \nabla_x f=\frac{1}{\eps} \mathcal{Q}(f).
\end{equation}
Here $f=f(t,x,v)$ is the probability density function (PDF) of time $t\geq 0$, position $x\in \mathbb{R}^d$, and velocity $v \in \mathbb{R}^d$. $\mathcal{Q}$ is the collision operator modeling the interaction between particles: $\mathcal{Q}$ could be the Boltzmann collision operator which is a nonlinear integral operator \cite{Cercignani}, or some simplified version mimicking the properties of the full Boltzmann operator, e.g., BGK \cite{BGK54}, ES-BGK \cite{Holway66}, or Shakov models \cite{Shakhov68}. In this context, $\eps$ is the so-called Knudsen number defined as the ratio of the mean free path and characteristic length. In practice, the value of $\eps$ characterizes the flow regime \cite{Struchtrup}: 1) $\eps \rightarrow 0$, Euler regime (the flow is well described by the compressible Euler equations, analogue of (\ref{con})); 2) $0<\eps \lesssim 0.01 $, Navier-Stokes regime (the flow is well described by the Navier-Stokes-Fourier (NSF) equations, analogue of (\ref{condiff})); 3) $0.01\lesssim \eps \lesssim 1$, transition regime (NSF equations fail, and one has to resort to extended macroscopic models or the original equation (\ref{kinetic})). Since (\ref{linearhyp}) and (\ref{kinetic}) share a common structure (non-stiff convection/transport plus stiff relaxation/collision), we conjecture on a uniform error bound similar to (\ref{error}) for smooth solutions, which implies that the IMEX-BDF schemes could capture the Navier-Stokes regime or even the transition regime with uniform accuracy. While proving this rigorously is currently out of reach due to the complexity of the problem, we demonstrate it numerically in the paper using the kinetic BGK equation.

To close this section, we mention that a parallel direction to our work is to design uniformly accurate methods for highly oscillatory equations, see the recent work \cite{CLMV19} and references therein. Although some ideas involved are of a similar flavor, the structure of equations in both fields are very different: our problem is of dissipative nature (endowed with an entropy), while the highly oscillatory problem is often conservative (endowed with a Hamiltonian).

The rest of the paper is organized as follows. In Section~{\ref{sec:BDF}, we briefly introduce the IMEX-BDF schemes employed in this paper. In Section~\ref{sec:uniform}, we establish our main result for the linear system (\ref{linearhyp}). To this end, we first study the regularity of the solution in time, then introduce the fully discrete scheme by adopting a Fourier spectral method in space. With these preparations, we prove in Section~\ref{subsec:stability} the uniform stability of the schemes which is facilitated by a newly introduced multiplier technique, and then in Section~\ref{subsec:accuracy} the uniform accuracy of the schemes by combining the stability and consistency results. In Section~\ref{sec:variable}, we extend our analysis to the variable coefficient case. In Section~\ref{sec:num}, we present several numerical studies including the verification of our theoretical results for the linear problem and investigation of some nonlinear hyperbolic and kinetic equations. The paper is concluded in Section~\ref{sec:con}.

\section{IMEX-BDF schemes for (\ref{linearhyp})}
\label{sec:BDF}

In this section, we briefly describe the IMEX-BDF schemes for the prototype problem (\ref{linearhyp}) and fix the notations for the following discussion.

Let $U^n$, $V^n$ denote the numerical solutions at time $t_n=T_0+n\Delta t$, where $T_0$ is the initial time, $n$ is a non-negative integer, and $\Delta t$ is the time step size, then the $q$-th order IMEX-BDF scheme for the system \eqref{linearhyp} reads
\begin{equation}\label{BDF0}\left\{\begin{split}
& \sum_{i=0}^q\alpha_i U^{n+i} + \Delta t \sum_{i=0}^{q-1}\gamma_i \partial_x V^{n+i} = 0, \\
& \sum_{i=0}^q\alpha_i V^{n+i} + \Delta t \sum_{i=0}^{q-1}\gamma_i \partial_x U^{n+i} = \frac{\beta\Delta t}{\eps}(bU^{n+q}-V^{n+q}), \\
\end{split}\right.\end{equation}
where the coefficients $\alpha=(\alpha_0,\dots,\alpha_q)$, $\gamma=(\gamma_0,\dots,\gamma_{q-1})$ and $\beta$ are given by
\begin{equation}
\beta = \frac{1}{\sum_{j=1}^q\frac{1}{j}},\quad \alpha(\zeta)= \beta\sum_{j=1}^q\frac{1}{j}\zeta^{q-j}(\zeta-1)^j=\sum_{i=0}^q \alpha_i\zeta^i,\quad \gamma(\zeta)= \beta(\zeta^q-(\zeta-1)^q)=\sum_{i=0}^{q-1} \gamma_i\zeta^i.
\end{equation}
The coefficients of IMEX-BDF schemes up to fourth order are provided in Table~\ref{coeff_table}. Note that IMEX-BDF1 is just the forward-backward Euler scheme.
\begin{table}[htp]
\centering
\begin{tabular}{|c|c|c|c|}
\hline 
$q$ & $\alpha$ & $\gamma$ & $\beta$ \\
\hline 
1 & $(-1,1)$ & 1 & 1 \\
2 & $\left(\frac{1}{3},-\frac{4}{3},1\right)$ & $\left(-\frac{2}{3},\frac{4}{3}\right)$ & $\frac{2}{3} $\\ 
3 & $\left(-\frac{2}{11},\frac{9}{11},-\frac{18}{11},1\right) $& $\left(\frac{6}{11},-\frac{18}{11},\frac{18}{11}\right)$ & $\frac{6}{11}$ \\
4 & $\left(\frac{3}{25},-\frac{16}{25},\frac{36}{25},-\frac{48}{25},1\right)$ &$ \left(-\frac{12}{25},\frac{48}{25},-\frac{72}{25},\frac{48}{25}\right)$ & $ \frac{12}{25} $\\
\hline 
\end{tabular}
\caption{Coefficients of IMEX-BDF schemes up to fourth order (cf.~\cite{HR07}).}
\label{coeff_table}
\end{table}

To initiate the scheme (\ref{BDF0}), one needs $2q$ starting values: $U^0, \dots, U^{q-1}$, $V^0, \dots, V^{q-1}$. In practice, they can be obtained using IMEX-RK schemes. With the proper starting values, the scheme can proceed by first solving for $U^{n+q}$ from the first equation of \eqref{BDF0} and then solving for $V^{n+q}$ from the second equation.

The above IMEX-BDF scheme is AP in the sense that, for initial data satisfying $V^i = bU^i,\,i=0,\dots,q-1$, it captures the asymptotic limit \eqref{con} when $\eps \rightarrow 0$ with fixed $\Delta t$. Indeed, as $\varepsilon \rightarrow 0$, the second equation of \eqref{BDF0} formally implies $V^{n+q}=bU^{n+q}$ for $n\geq 0$, hence $V^i=bU^i$ for all $i\ge0$. Substituting this into the first equation of \eqref{BDF0} yields
\begin{equation}\label{limit_sch}
 \sum_{i=0}^q\alpha_i U^{n+i} + \Delta t \sum_{i=0}^{q-1}\gamma_i \partial_x (bU^{n+i})= 0,
\end{equation}
i.e., in the limit $\varepsilon \rightarrow 0$, the scheme becomes a consistent $q$-th order explicit multistep scheme for \eqref{con}.

\begin{remark}
Here for simplicity we only present the scheme for the linear problem (\ref{linearhyp}). If the relaxation term is nonlinear, e.g., $bu$ on the right-hand side of (\ref{linearhyp}) is replaced by a nonlinear function $f(u)$, the scheme works equally. In fact, due to the special structure of the equation, one can still obtain $U^{n+q}$ first and then $V^{n+q}$, and no iteration is needed, i.e., a nonlinear implicit scheme can be implemented explicitly. Similar idea applies to the nonlinear kinetic equations, for instance, the BGK equation, see Appendix~\ref{appendix:A} for a brief description.
\end{remark}

\section{Uniform accuracy of the IMEX-BDF schemes}
\label{sec:uniform}

In this section, we prove the uniform accuracy of the IMEX-BDF schemes for the linear hyperbolic relaxation system (\ref{linearhyp}). To this end, we need to employ the spatial discretization as well. For simplicity, we assume that the spatial domain is $[0,2\pi]$ with periodic boundary condition and adopt the Fourier-Galerkin spectral method. 

The proof of uniform accuracy consists of three steps. First, we study the regularity of \eqref{linearhyp}, especially for high order time derivatives of $u$ and $v$, which is necessary for high order consistency of the numerical scheme. Next, we establish the uniform stability of the fully discrete scheme by using energy estimates with the multiplier technique. A new class of multipliers is introduced to overcome the difficulty coming from the special structure of \eqref{linearhyp}. Finally, we prove the uniform accuracy by combining the consistency and stability results.

In the following, all the integrals without range refer to $\int_0^{2\pi}\cdot \,\rd{x}$; function norms $\|\cdot \|$ without subscript refer to $L^2$ norm in $x$; and $C$ denotes a generic constant independent of $\eps$.

\subsection{Regularity estimate}

Hyperbolic relaxation systems, especially the nonlinear ones, have been studied throughly in the 90's, where the main concern is to justify rigorously the convergence to the zero relaxation limit (see \cite{Natalini98} for a survey). Here our purpose is different as we need the uniform-in-$\eps$ regularity in time to prove the consistency of high order time discretization. Therefore, we give a self-contained proof in this subsection.

To study the regularity of \eqref{linearhyp}, we first reformulate it into a new form, similar to the micro-macro decomposition in kinetic theory \cite{LY04, BLM08}. Introducing a new variable $w = v-bu$, one can rewrite \eqref{linearhyp} as
\begin{equation}\label{eq1}\left\{\begin{split}
& \partial_t u + \partial_x (bu+w) = 0, \\
& \partial_t w + \partial_x ((1-b^2)u-bw) = -\frac{1}{\eps}w.
\end{split}\right.\end{equation}
Multiplying the two equations of \eqref{eq1} by $u$ and $w$ respectively and integrating in $x$, we get
\begin{equation}
\partial_t \frac{1}{2}\|u\|^2 + \int u\partial_x w \rd{x} = 0 ,
\end{equation}
and
\begin{equation}
\partial_t \frac{1}{2}\|w\|^2 + (1-b^2)\int w\partial_x u \rd{x} = -\frac{1}{\eps}\|w\|^2.
\end{equation}
Then a linear combination of the above two relations gives the energy estimate
\begin{equation}\label{energy}
\partial_t \frac{1}{2}((1-b^2)\|u\|^2+\|w\|^2) = -\frac{1}{\eps}\|w\|^2 \le 0,
\end{equation}
which implies
\begin{equation}\label{energyd}
\frac{1}{2}((1-b^2)\|u(t)\|^2+\|w(t)\|^2) \le \frac{1}{2}((1-b^2)\|u(0)\|^2+\|w(0)\|^2).
\end{equation}
Notice that the condition $|b|<1$ guarantees the equivalence of $\|u\|^2+\|w\|^2$ and the Lyapunov functional $(1-b^2)\|u\|^2+\|w\|^2$.

The regularity estimate for \eqref{eq1} is stated as follows.

\begin{theorem}\label{thm_reg}
For any integer $s\ge 0$, assume 
\begin{equation}\label{thm_reg_0}
\|u(0)\|_{H^s}^2 + \|w(0)\|_{H^s}^2 =: E_0 < \infty.
\end{equation}
Then for all $t\ge s\eps \log(1/\eps)$, the solution to \eqref{eq1} satisfies
\begin{equation} \label{energy1}
\|u(t)\|_{H^s}^2 + \|w(t)\|_{H^s}^2 \le CE_0,
\end{equation}
also
\begin{equation}\label{thm_reg_1}
\|\partial_t^{r_1}\partial_x^{r_2} u(t)\|^2 + \|\partial_t^{r_1}\partial_x^{r_2} w(t)\|^2 \le CE_0,\quad r_1+r_2 \le s,
\end{equation}
and
\begin{equation}\label{thm_reg_2}
\|\partial_t^{r_1}\partial_x^{r_2} w(t)\|^2 \le CE_0\eps^2, \quad r_1+r_2 \le s-1.
\end{equation}
\end{theorem}

\begin{proof}
First of all, for any integer $s\geq 0$, $(\partial_x^s u,\partial_x^s w)$ satisfies the same system \eqref{eq1}, thus also has the energy estimate \eqref{energyd}. This implies (\ref{energy1}).

We then prove that \eqref{thm_reg_0} implies \eqref{thm_reg_1} by induction on $s$. The case $s=0$ clearly follows from \eqref{energy1}. Now assume \eqref{thm_reg_0} implies \eqref{thm_reg_1} with $s$ replaced by $s-1$ in both equations. We aim to prove that \eqref{thm_reg_0} implies \eqref{thm_reg_1} in the case of $s$. Since for any $0\le r \le s-1$, $(\partial_t\partial_x^r u,\partial_t\partial_x^r w)$ satisfies the same system \eqref{eq1}, by (\ref{energy}) one has
\begin{equation}\label{energy2}\begin{split}
& \partial_t \frac{1}{2}\left((1-b^2)\|\partial_t \partial_x^r u\|^2+\|\partial_t \partial_x^r w\|^2\right) = -\frac{1}{\eps}\|\partial_t \partial_x^r w\|^2  \\
= & -\frac{1}{\eps}\left((1-b^2)\|\partial_t \partial_x^r u\|^2+\|\partial_t \partial_x^r w\|^2\right) + \frac{1}{\eps}(1-b^2)\|\partial_t \partial_x^r u\|^2.
\end{split}\end{equation}
Notice that $\partial_t \partial_x^r u = -b\partial_x^{r+1} u - \partial_x^{r+1} w$. Thus by \eqref{energy1},
\begin{equation}
\|\partial_t \partial_x^r u(t)\|^2 \le CE_0.
\end{equation}
Similarly, from $\partial_t \partial_x^r w = -(1-b^2)\partial_x^{r+1}u+b\partial_x^{r+1}w -\frac{1}{\eps}\partial_x^rw$ one has
\begin{equation}
\|\partial_t \partial_x^r w(t)\|^2 \le \frac{1}{\eps^2}CE_0.
\end{equation}
Then by \eqref{energy2},
\begin{equation}\begin{split}
& \frac{1}{2}\left((1-b^2)\|\partial_t \partial_x^r u(t)\|^2+\|\partial_t \partial_x^r w(t)\|^2\right)\\
= & e^{-2t/\eps}\frac{1}{2}((1-b^2)\|\partial_t \partial_x^r u(0)\|^2+\|\partial_t \partial_x^r w(0)\|^2)  + \frac{1}{\eps}(1-b^2)\int_0^t e^{-2(t-\tau)/\eps}\|\partial_\tau \partial_x^r u(\tau)\|^2\rd{\tau} \\
\le & \left(e^{-2t/\eps}\frac{1}{\eps^2} +1\right)CE_0.
\end{split}\end{equation}
This implies
\begin{equation}\label{energy3}
\|\partial_t \partial_x^r u(t_1)\|^2+\|\partial_t \partial_x^r w(t_1)\|^2 \le CE_0,\quad t_1 = \eps \log(1/\eps).
\end{equation}
Define
\begin{equation}
\tilde{u}(t) = \partial_t u(t+t_1),\quad \tilde{w}(t) = \partial_t w(t+t_1).
\end{equation}
Then $(\tilde{u},\tilde{w})$ also satisfies \eqref{eq1}, and \eqref{energy3} means
\begin{equation}
\|\tilde{u}(0)\|_{H^{s-1}}^2 + \|\tilde{w}(0)\|_{H^{s-1}}^2 \le C E_0 .
\end{equation}
By the induction hypothesis, we have the estimate
\begin{equation}
\|\partial_t^{r_1}\partial_x^{r_2} \tilde{u}(t)\|^2 + \|\partial_t^{r_1}\partial_x^{r_2} \tilde{w}(t)\|^2 \le CE_0,\quad r_1+r_2 \le s-1,
\end{equation}
for all $t\ge (s-1)\eps \log(1/\eps)$. In view of the definition of $(\tilde{u},\tilde{w})$, this together with \eqref{energy1} implies \eqref{thm_reg_1}.

Finally, \eqref{thm_reg_2} follows from \eqref{thm_reg_1} using the relation $w = -\eps(\partial_t w + \partial_x ((1-b^2)u-bw))$.
\end{proof}

\subsection{Spatial discretization and fully discrete scheme}

For the spatial discretization, we apply the Fourier-Galerkin spectral method. Consider the space of trigonometric polynomials of degree up to $N$:
\begin{equation}\label{P}
\mathbb{P}_N=\text{span} \{ e^{ikx} | -N\leq k \leq N\},
\end{equation} 
equipped with inner product
\begin{equation}
\langle f, g \rangle=\frac{1}{2\pi }\int  f\bar{g}\,\rd{x}.
\end{equation}
For a given function $f(x)$, its projection $\mathcal{P}_Nf$ is defined as
\begin{equation}
\mathcal{P}_N f(x) =\sum_{k=-N}^N \hat{f}_k e^{ikx} \in \mathbb{P}_N, \quad \hat{f}_k=\langle f,e^{ikx}\rangle.
\end{equation}
For the projection operator, we have the following basic facts, whose proof can be found in \cite{HGG07} for instance.

\begin{lemma} \label{adjoint}
$\mathcal{P}_N$ is self-adjoint, i.e., for any functions $f(x)$ and $g(x)$, there holds
\begin{equation}
\langle \mathcal{P}_Nf,g\rangle =\langle f,\mathcal{P}_Ng\rangle.
\end{equation}
\end{lemma}

\begin{lemma} \label{Fourier1}
For any $2\pi$-periodic function $f(x)\in H^s[0,2\pi]$, there holds
\begin{equation}
\|(I-\mathcal{P}_N)f\|\leq \frac{1}{N^s}\|f\|_{H^s}.
\end{equation}
\end{lemma}
\begin{lemma} \label{Fourier2}
For any $\phi(x)\in \mathbb{P}_N$, there holds
\begin{equation}
\| \phi^{(s)}\|\leq N^{s}\|\phi\|.
\end{equation}
\end{lemma}

The Fourier-Galerkin spectral method for (\ref{eq1}) seeks to approximate $u$, $w$ as
\begin{equation} \label{FG}
u(t,x) \approx \sum_{k=-N}^N u_k(t)e^{ikx} =: u_N(t,x), \quad w(t,x) \approx \sum_{k=-N}^N w_k(t)e^{ikx} =: w_N(t,x).
\end{equation}
Substituting (\ref{FG}) into \eqref{eq1} and conducting the Galerkin projection yields
\begin{equation} \label{FG1}
\left\{\begin{split}
& \partial_t u_N + \partial_x (bu_N+w_N) = 0, \\
& \partial_t w_N + \partial_x ((1-b^2)u_N-bw_N) = -\frac{1}{\eps}w_N,
\end{split}\right.\end{equation}
i.e., $u_N,w_N$ still satisfy \eqref{eq1}. It is worth mentioning that in the actual numerical method, the coefficients $u_k(t)$ and $w_k(t)$ are the sought after quantities. They satisfy the same equation (\ref{FG1}) and the initial condition $u_k(0)=\mathcal{P}_N u(0,x)$, $w_k(0)=\mathcal{P}_N w(0,x)$. Applying the $q$-th order IMEX-BDF scheme to (\ref{FG1}), similarly as (\ref{BDF0}), yields the fully discrete scheme
\begin{equation}\label{BDF}\left\{\begin{split}
& \sum_{i=0}^q\alpha_i U_N^{n+i} + \Delta t \sum_{i=0}^{q-1}\gamma_i \partial_x (bU_N^{n+i}+W_N^{n+i}) = 0, \\
& \sum_{i=0}^q\alpha_i W_N^{n+i} + \Delta t \sum_{i=0}^{q-1}\gamma_i \partial_x ((1-b^2)U_N^{n+i}-bW_N^{n+i}) = -\frac{\beta\Delta t}{\eps}W_N^{n+q}, \\
\end{split}\right.\end{equation}
where $U_N^n, W_N^n$ are the fully discrete solutions. For the starting values, we assume they can be obtained with very high accuracy (in time) using IMEX-RK schemes so that the error only comes from the spatial discretization, i.e., we assume
\begin{equation}
U_N^i=\mathcal{P}_N u(t_i,x), \quad  W_N^i=\mathcal{P}_N w(t_i,x), \quad i=0,\dots,q-1,
\end{equation}
where $u$ and $w$ are the exact solutions. 


\subsection{Uniform stability}
\label{subsec:stability}

To analyze the stability of \eqref{BDF}, we adopt the multiplier technique. We mention that this technique appeared already in early works \cite{Dahlquist78, NO81} but it was only recently used in the analysis of BDF schemes for parabolic equations, see for instance \cite{LMV13, AL15}. When it comes to our problem (\ref{linearhyp}), the existing multipliers, including the recently proposed ones \cite{AK16}, cannot be applied due to the special convection-relaxation structure of the equation. Therefore, we invent a new multiplier as stated in the following lemma.

\begin{lemma}\label{lem_G}
For $q=1,2,3,4$, there exist positive-definite Hermitian form $G(u_1,\dots,u_q)=\sum_{i,j=1}^q g_{ij}u_i\bar{u}_j$, semi-positive-definite Hermitian form $A(u_1,\dots,u_{q-1})=\sum_{i,j=1}^{q-1} a_{ij}u_i\bar{u}_j$, linear forms $\sum_{i=1}^{q-1}\eta_i u_i$, $\sum_{i=1}^q c_i u_i$, where all coefficients involved are real, and real constants $d_1>0$ and $d_2$, such that
\begin{equation}\label{lem_G_1}\begin{split}
& \Re\left(\Big(\bar{u}_q-\sum_{j=1}^{q-1}\eta_j \bar{u}_j\Big)\sum_{i=0}^q\alpha_i u_i\right) =  G(u_1,\dots,u_q)-G(u_0,\dots,u_{q-1}) + d_1\left|u_q-\sum_{i=1}^{q-1}\eta_i u_i - d_2\sum_{i=0}^{q-1}\gamma_i u_i\right|^2 ,
\end{split}\end{equation}
and
\begin{equation}\label{lem_G_2}
\Re\left(\Big(\bar{u}_q-\sum_{i=1}^{q-1}\eta_i \bar{u}_i\Big)u_q\right) = A(u_2,\dots,u_q) - A(u_1,\dots,u_{q-1}) + \left|\sum_{i=1}^q c_i u_i\right|^2,
\end{equation}
hold for any $u_0,u_1,\dots,u_q\in \mathbb{C}$, where $\Re$ denotes the real part of a complex number.
\end{lemma}


\begin{proof}
First of all, it is easy to see that to prove \eqref{lem_G_1} and \eqref{lem_G_2} it suffices to show the following 
equalities hold for any real numbers $u_0,u_1,\dots,u_q$:
\begin{equation}\label{lem_G_11}\begin{split}
& \Big(u_q-\sum_{j=1}^{q-1}\eta_j u_j\Big)\sum_{i=0}^q\alpha_i u_i=  G(u_1,\dots,u_q)-G(u_0,\dots,u_{q-1}) + d_1\left (u_q-\sum_{i=1}^{q-1}\eta_i u_i - d_2\sum_{i=0}^{q-1}\gamma_i u_i\right)^2 ,
\end{split}\end{equation}
and
\begin{equation}\label{lem_G_21}
\Big(u_q-\sum_{i=1}^{q-1}\eta_i u_i\Big)u_q = A(u_2,\dots,u_q) - A(u_1,\dots,u_{q-1}) + \left(\sum_{i=1}^q c_i u_i\right)^2.
\end{equation}

For $q=1$, the coefficients are given by
\begin{equation}
g_{11} = \frac{1}{2}, \quad d_1 = \frac{1}{2}, \quad d_2=1, \quad c_1 = 1.
\end{equation}
For $q=2$, the coefficients are given by
\begin{equation}
\eta_1 = 0,\quad g_{11} = \frac{1}{6},\quad g_{22} = \frac{5}{6},\quad g_{12} = -\frac{1}{3}, \quad d_1 = \frac{1}{6},\quad d_2 = \frac{3}{2},\quad a_{11}=0,\quad c_1=0, \quad  c_2 = 1.
\end{equation}
For these two cases, \eqref{lem_G_11}, \eqref{lem_G_21}, as well as the (semi-)positive-definiteness of $G$ and $A$ can be checked easily. In fact, these numbers are well-known, see for instance \cite{HW}.

For $q=3$, the coefficients are given by irrational numbers\footnote{These values are obtained by solving certain algebraic systems using the symbolic toolbox of MATLAB. They serve as one admissible set of solutions and we do not claim their uniqueness and optimality. Same can be said for the values provided in Appendix~\ref{appendix:B}  for $q=4$.}:
\begin{equation}\begin{split}
& g_{11} = \frac{\sqrt{30}}{187} + \frac{8}{187},\quad g_{22} = \frac{\sqrt{30}}{34} + \frac{95}{187},\quad g_{33} = \frac{\sqrt{30}}{22} + \frac{7}{11}, \\
& g_{12} = -\frac{3\sqrt{30}}{187} - \frac{24}{187},\quad g_{13} = \frac{3\sqrt{30}}{187} + \frac{24}{187},\quad g_{23} = -\frac{6\sqrt{30}}{187} - \frac{9}{17}, \\
& \eta_{1} = \frac{\sqrt{30}}{17} - \frac{9}{17},\quad \eta_{2} = -\frac{2\sqrt{30}}{17} + \frac{18}{17}, \quad d_{1} = -\frac{\sqrt{30}}{22} + \frac{4}{11},\quad d_{2} =  \frac{11\sqrt{30}}{102} + \frac{44}{51}.
\end{split}\end{equation}
\begin{equation}\begin{split}
a_{11} = & \Big( (5576634533850159812-1018149509409713088\sqrt{30})z_*^6 \\
& + (-827564175794699168+151091855378090876\sqrt{30})z_*^4 \\
& + (1317402834013463958-240523749880736072\sqrt{30})z_*^2 \\
& + (-150042582540986748+27393902345391585\sqrt{30}) \Big) \\
& \Big/ (-1011078865344820767 + 184596900652059714\sqrt{30}), \\
a_{12} = & -\Big( (20162952-3576664\sqrt{30})z_*^6 + (-11669820+2036872\sqrt{30})z_*^4 \\
& + (9540978-1747158\sqrt{30})z_*^2  + (-4604391+840294\sqrt{30}) \Big)  \Big/ (-7213644 + 1314780\sqrt{30}), \\
a_{22} = & 1-z_*^2, \\
c_{1} = & \Big( (1454248-235824\sqrt{30})z_*^7  + (-268192+25432\sqrt{30})z_*^5 \\
& + (312732-61336\sqrt{30})z_*^3  + (-37632+6618\sqrt{30})z_* \Big)  \Big/ (-106083 + 19335\sqrt{30}), \\
c_{2} = & \Big( -39304z_*^7  + (28900+1156\sqrt{30})z_*^5 \\
& + (-8466+1904\sqrt{30})z_*^3 + (4617-819\sqrt{30})z_* \Big) \Big/ (-3078 + 546\sqrt{30}), \\
c_3 = & z_*, 
\end{split}\end{equation}
where $z_*$ is the unique real root of the polynomial
\begin{equation}\begin{split}
\phi(z) = & (-49444432+8018016\sqrt{30})z^8 - (-9118528+864688\sqrt{30})z^6 + (-10632888+2085424\sqrt{30})z^4 \\
&  - (-1279488+225012\sqrt{30})z^2 + (-1534797+280098\sqrt{30})
\end{split}\end{equation}
that lies in the interval $(0.106,0.107)$.\footnote{The approximate value of $z_*$ is given by
\begin{equation*}
z_*\approx 0.10618875349491630708729892823342.
\end{equation*}} \eqref{lem_G_11} and \eqref{lem_G_21} can be checked directly (by plugging in the coefficients and comparing like terms on both hand sides). To check the positive-definiteness of $G$, we compute the characteristic polynomial of $G$:
\begin{equation}
\chi_G(\lambda) = \lambda^3 + \left(-\frac{15\sqrt{30}}{187}-\frac{222}{187}\right)\lambda^2 + \left(\frac{312\sqrt{30}}{34969}+\frac{4533}{69938}\right)\lambda + \left(-\frac{45\sqrt{30}}{13078406}-\frac{656}{6539203}\right),
\end{equation} 
and it is clear that $\chi_G(\lambda)<0$ for $\lambda\le 0$. Therefore $\chi_G$ has only positive roots.\footnote{In fact, the smallest eigenvalue of $G$ is approximately
\begin{equation*}
\lambda_1(G) \approx 0.001064408628491745818998719988681.
\end{equation*}
}
Since $A$ is a quadratic form of two variables, to check the positive-definiteness of $A$, it suffices to check $a_{22}>0$ and $a_{11}a_{22}-a_{12}^2 > 0$, for the root $z_*$ of $\phi(z)$: the first fact is clear; to check the second fact, we compute
\begin{equation}
a_{11}a_{22}-a_{12}^2 = \left(\frac{161}{578}+\frac{26\sqrt{30}}{289}\right) + \left(\frac{96}{17}+\frac{53\sqrt{30}}{51}\right) z_*^2 + \left(\frac{85}{3}+\frac{16\sqrt{30}}{3}\right) z_*^4 + \left(-\frac{112}{3}-\frac{20\sqrt{30}}{3}\right) z_*^6,
\end{equation}
by using $\phi(z_*)=0$. Since $\frac{161}{578}+\frac{26\sqrt{30}}{289} > 0.77$ and $\frac{112}{3}+\frac{20\sqrt{30}}{3} < 74$, one has
\begin{equation}
a_{11}a_{22}-a_{12}^2 > 0.77 - 74\cdot 0.107^6 > 0,
\end{equation}
by using $0.106<z_*<0.107$.\footnote{In fact, the smallest eigenvalue of $A$ is approximately
\begin{equation*}
\lambda_1(A) \approx 0.78260015292507185414401336030175.
\end{equation*}
}

For $q=4$, we list in Appendix~\ref{appendix:B} an approximate choice of coefficients which satisfy \eqref{lem_G_11} and \eqref{lem_G_21} up to an error of $10^{-31}$, with the quadratic forms $G$ and $A$ being positive-definite. One can argue rigorously as that for $q=3$. We omit the details.
\end{proof}

\begin{remark}
We did not consider the case with $q>4$ because for kinetic equations \eqref{kinetic} that we are mostly interested in, the spatial error often dominates even using high order methods such as the fifth order WENO method \cite{Shu98} and usually up to third order accuracy in time would suffice (see also the numerical results in Section~\ref{subsec:BGK}).
\end{remark}

\begin{remark}
Compared with the multipliers constructed in \cite{NO81} for $q\le 5$, the main feature of Lemma \ref{lem_G} is a precise quantification of the last square term in \eqref{lem_G_1}: it is exactly the square of a linear combination of the multiplier $u_q-\sum_{i=1}^{q-1}\eta_i u_i$ and the explicit part $\sum_{i=0}^{q-1}\gamma_i u_i$ appearing in the IMEX-BDF scheme. This will be essential to the proof of uniform stability, for handling the contribution from the convection term.
\end{remark}

Next we state our result on the uniform stability of the IMEX-BDF scheme \eqref{BDF}:

\begin{theorem}[Uniform stability of IMEX-BDF schemes]\label{thm_stab}
For $q=1,2,3,4$, under the CFL condition $\Delta t \le c_{CFL}/N^2$ for any constant $c_{CFL}>0$, the IMEX-BDF scheme \eqref{BDF} is uniformly stable, in the sense that
\begin{equation} \label{uniform}
\|U_N^n\|^2 + \|W_N^n\|^2 \le C\sum_{i=0}^{q-1}\left(\|U_N^i\|^2+\Big(1+\frac{\Delta t}{\eps}\Big)\|W_N^i\|^2\right),
\end{equation}
for any $n$ such that $t_n=T_0+n\Delta t\leq T$, where $C$ is a constant independent of $\varepsilon$, $N$ and $\Delta t$.
\end{theorem}

\begin{proof}
In this proof we follow the notations from Lemma \ref{lem_G}. Denote
\begin{equation}\begin{split}
& G^n_U = \int G(U_N^n,\dots,U_N^{n+q-1})\rd{x},\quad G^n_W = \int G(W_N^n,\dots,W_N^{n+q-1})\rd{x},\\
& A^n_W = \int A(W_N^{n+1},\dots,W_N^{n+q-1})\rd{x}.
\end{split}\end{equation}

Multiplying the first equation of \eqref{BDF} by $\bar{U}_N^{n+q}-\sum_{i=1}^{q-1}\eta_i \bar{U}_N^{n+i}$, using Lemma \ref{lem_G} and integrating in $x$ gives
\begin{equation}\label{GU}\begin{split}
0 = & G^{n+1}_U-G^n_U + d_1\left\|U_N^{n+q}-\sum_{i=1}^{q-1}\eta_i U_N^{n+i} - d_2\sum_{i=0}^{q-1}\gamma_i U_N^{n+i}\right\|^2 \\
& + b\Delta t \Re\int \left(\bar{U}_N^{n+q}-\sum_{j=1}^{q-1}\eta_j \bar{U}_N^{n+j}\right)\partial_x \sum_{i=0}^{q-1}\gamma_i  U_N^{n+i} \rd{x} \\
& + \Delta t \Re\int \left(\bar{U}_N^{n+q}-\sum_{j=1}^{q-1}\eta_j \bar{U}_N^{n+j}\right)\partial_x \sum_{i=0}^{q-1}\gamma_i  W_N^{n+i} \rd{x} \\
= & G^{n+1}_U-G^n_U + d_1\left\|U_N^{n+q}-\sum_{i=1}^{q-1}\eta_i U_N^{n+i} - d_2\sum_{i=0}^{q-1}\gamma_i U_N^{n+i}\right\|^2 \\
& + b\Delta t \Re\int \left(\bar{U}_N^{n+q}-\sum_{j=1}^{q-1}\eta_j \bar{U}_N^{n+j}-d_2\sum_{j=0}^{q-1}\gamma_j  \bar{U}_N^{n+j}\right)\partial_x \sum_{i=0}^{q-1}\gamma_i  U_N^{n+i} \rd{x} \\
& + \Delta t \Re\int \left(\bar{U}_N^{n+q}-\sum_{j=1}^{q-1}\eta_j \bar{U}_N^{n+j}-d_2\sum_{j=0}^{q-1}\gamma_j  \bar{U}_N^{n+j}\right)\partial_x \sum_{i=0}^{q-1}\gamma_i  W_N^{n+i} \rd{x} \\
& + d_2\Delta t \Re\int \sum_{j=0}^{q-1}\gamma_j \bar{U}_N^{n+j}\partial_x \sum_{i=0}^{q-1}\gamma_i  W_N^{n+i} \rd{x},
\end{split}\end{equation}
where the new terms added in the last equality are actually zero due to an integration by parts and periodic boundary condition. Similar treatment for the second equation of \eqref{BDF} gives
\begin{equation}\label{GW}\begin{split}
 & G^{n+1}_W-G^n_W + d_1\left\|W_N^{n+q}-\sum_{i=1}^{q-1}\eta_i W_N^{n+i} - d_2\sum_{i=0}^{q-1}\gamma_i W_N^{n+i}\right\|^2 \\
& - b\Delta t \Re\int \left(\bar{W}_N^{n+q}-\sum_{j=1}^{q-1}\eta_j \bar{W}_N^{n+j}-d_2\sum_{j=0}^{q-1}\gamma_j \bar{W}_N^{n+j}\right)\partial_x \sum_{i=0}^{q-1}\gamma_i  W_N^{n+i} \rd{x} \\
& + (1-b^2)\Delta t \Re\int \left(\bar{W}_N^{n+q}-\sum_{j=1}^{q-1}\eta_j \bar{W}_N^{n+j}-d_2\sum_{j=0}^{q-1}\gamma_j  \bar{W}_N^{n+j}\right)\partial_x \sum_{i=0}^{q-1}\gamma_i  U_N^{n+i} \rd{x} \\
& + (1-b^2)d_2\Delta t \Re\int \sum_{j=0}^{q-1}\gamma_j  \bar{W}_N^{n+j}\partial_x \sum_{i=0}^{q-1}\gamma_i  U_N^{n+i} \rd{x} \\
= & -\frac{\beta\Delta t}{\eps}\left(A^{n+1}_W-A^n_W + \left\|\sum_{i=1}^q c_i W_N^{n+i}\right\|^2\right). 
\end{split}\end{equation}

Notice that for any $\kappa>0$, using Young's inequality,
\begin{equation}\begin{split}
& \left|\Delta t\Re\int \left(\bar{U}_N^{n+q}-\sum_{j=1}^{q-1}\eta_j \bar{U}_N^{n+j}-d_2\sum_{j=0}^{q-1}\gamma_j  \bar{U}_N^{n+j}\right)\partial_x \sum_{i=0}^{q-1}\gamma_i  U_N^{n+i} \rd{x}\right| \\
\le & \kappa\left\|U_N^{n+q}-\sum_{i=1}^{q-1}\eta_i U_N^{n+i}-d_2\sum_{i=0}^{q-1}\gamma_i  U_N^{n+i}\right\|^2 + \frac{\Delta t^2}{4\kappa}\left\|\partial_x \sum_{i=0}^{q-1}\gamma_i  U_N^{n+i}\right\|^2 \\
\le & \kappa\left\|U_N^{n+q}-\sum_{i=1}^{q-1}\eta_i U_N^{n+i}-d_2\sum_{i=0}^{q-1}\gamma_i  U_N^{n+i}\right\|^2 + \frac{\Delta t^2}{4\kappa}C \sum_{i=0}^{q-1} \left\|\partial_x  U_N^{n+i}\right\|^2,
\end{split}\end{equation}
and similarly for the other three integrals of this form in (\ref{GU}) and (\ref{GW}). Then we take $(1-b^2)\times \eqref{GU} + \eqref{GW}$, and by estimating these integrals with $\kappa$ small (in terms of $b$ and $d_1$), we may absorb the terms with coefficient $\kappa$ by the two terms with coefficient $d_1$, and obtain
\begin{equation}\begin{split}
& ((1-b^2)G^{n+1}_U+G^{n+1}_W)-((1-b^2)G^{n}_U+G^{n}_W) \\
\le & C\Delta t^2\sum_{i=0}^{q-1}\left(\left\|\partial_x  U_N^{n+i}\right\|^2+\left\|\partial_x  W_N^{n+i}\right\|^2\right) -\frac{\beta\Delta t}{\eps}\left(A_W^{n+1}-A_W^n + \left\|\sum_{i=1}^q c_i W_N^{n+i}\right\|^2\right),
\end{split}\end{equation}
where we used the fact that 
\begin{equation}
\Re\int \sum_{j=0}^{q-1}\gamma_j \bar{U}_N^{n+j}\partial_x \sum_{i=0}^{q-1}\gamma_i  W_N^{n+i} \rd{x}+\Re\int \sum_{j=0}^{q-1}\gamma_j \bar{W}_N^{n+j}\partial_x \sum_{i=0}^{q-1}\gamma_i  U_N^{n+i} \rd{x}=0.
\end{equation}
Therefore,
\begin{equation}\begin{split}
& ((1-b^2)G^{n+1}_U+G^{n+1}_W+\frac{\beta\Delta t}{\eps}A_W^{n+1})-((1-b^2)G^{n}_U+G^{n}_W+\frac{\beta\Delta t}{\eps}A_W^{n}) \\
\le & C\Delta t^2\sum_{i=0}^{q-1}\left(\left\|\partial_x  U_N^{n+i}\right\|^2 + \left\|\partial_x  W_N^{n+i}\right\|^2 \right).
\end{split}\end{equation}

Using Lemma~\ref{Fourier2} and the CFL condition, one has
\begin{equation}
\|\partial_x U_N^n\| \le N\|U_N^n\| \le \sqrt{\frac{c_{CFL}}{\Delta t}}\|U_N^n\|.
\end{equation}
The same estimate holds for $W_N^n$. Therefore,
\begin{equation} \label{qq}
\begin{split}
& ((1-b^2)G^{n+1}_U+G^{n+1}_W+\frac{\beta\Delta t}{\eps}A_W^{n+1})-((1-b^2)G^{n}_U+G^{n}_W+\frac{\beta\Delta t}{\eps}A_W^{n}) \\
\le & Cc_{CFL}\Delta t\sum_{i=0}^{q-1}\left(\left\| U_N^{n+i}\right\|^2 + \left\| W_N^{n+i}\right\|^2\right) .
\end{split}\end{equation}

On the other hand, the positive-definiteness of $G$ guarantees the existence of $\lambda_{\max}(q)\geq \lambda_{\min}(q)>0$ such that
\begin{equation}\label{lambda}
\lambda_{\min}\sum_{i=0}^{q-1}\left\| U_N^{n+i}\right\|^2 \leq G^n_U \leq \lambda_{\max}\sum_{i=0}^{q-1}\left\| U_N^{n+i}\right\|^2,\quad \lambda_{\min}\sum_{i=0}^{q-1}\left\| W_N^{n+i}\right\|^2 \leq G^n_W \leq \lambda_{\max}\sum_{i=0}^{q-1}\left\| W_N^{n+i}\right\|^2.
\end{equation}
The semi-positive-definiteness of $A$ guarantees the existence of $\eta_{\max}(q)>0$ such that
\begin{equation}
0\leq A_W^n\leq  \eta_{\max} \sum_{i=1}^{q-1}\left\| W_N^{n+i}\right\|^2.
\end{equation}
Define 
\begin{equation}\label{En}
E^n = (1-b^2)G^{n}_U+G^{n}_W+\frac{\beta\Delta t}{\eps}A^{n}_W,
\end{equation}
then
\begin{equation}
(1-b^2)\lambda_{\min}\sum_{i=0}^{q-1}\left(\left\| U_N^{n+i}\right\|^2 + \left\| W_N^{n+i}\right\|^2\right)\leq (1-b^2)G_U^n+G_W^n\leq E^n,
\end{equation}
i.e.,
\begin{equation}
\sum_{i=0}^{q-1}\left(\left\| U_N^{n+i}\right\|^2 + \left\| W_N^{n+i}\right\|^2\right)\leq CE^n.
\end{equation}
Also
\begin{equation}
(1-b^2)G_U^n+G_W^n+\frac{\beta\Delta t}{\eps}A^{n}_W \leq \lambda_{\max}\sum_{i=0}^{q-1}\left(\left\| U_N^{n+i}\right\|^2 + \left\| W_N^{n+i}\right\|^2\right)+\frac{\beta \Delta t}{\eps}\eta_{\max}\sum_{i=1}^{q-1}\left\| W_N^{n+i}\right\|^2,
\end{equation}
i.e.,
\begin{equation}
E^n\leq  C\sum_{i=0}^{q-1}\left(\|U_N^{n+i}\|^2+\Big(1+\frac{\Delta t}{\eps}\Big)\|W_N^{n+i}\|^2\right).
\end{equation}

Therefore, (\ref{qq}) implies
\begin{equation}
E^{n+1} \le (1+C\Delta t)E^n,
\end{equation}
which gives
\begin{equation}
E^n \le (1+C\Delta t)^n E^0 \le \exp(CT)E^0.
\end{equation}
Finally, the conclusion follows by noticing that
\begin{equation}
\|U_N^n\|^2 + \|W_N^n\|^2 \le CE^n, \quad E^0 \le C\sum_{i=0}^{q-1}\left(\|U_N^i\|^2+\Big(1+\frac{\Delta t}{\eps}\Big)\|W_N^i\|^2\right).
\end{equation}
\end{proof}

\begin{remark}
For $q=1,2$, the same estimate \eqref{uniform} holds with the $\frac{\Delta t}{\eps}$ term removed. In view of \eqref{En}, this is a consequence of the fact that for $q=1,2$, $A$ can be taken as zero in Lemma \ref{lem_G}.
\end{remark}

\begin{remark}
The CFL condition $\Delta t\leq c_{CFL}/N^2$ is the standard stability condition one will obtain when using the forward Euler coupled with Fourier spectral method to solve the hyperbolic system \eqref{linearhyp} without relaxation term \cite{HGG07}. 
\end{remark}

\subsection{Uniform accuracy}
\label{subsec:accuracy}

In this subsection, we establish our main result on the uniform accuracy of the IMEX-BDF schemes. To this end, we distinguish two kinds of initial conditions.

{\bf Case 1.} For $q=1,2,3,4$, if the initial data is consistent up to order $q$, in the sense that
\begin{equation}\label{thm_accu_cons}\begin{split}
& \|\partial_t^{q+1}u(0)\|_{H^1} + \|\partial_t^{q+1}w(0)\|_{H^1} \le C, \\
& \|\partial_t^{q}u(0)\|_{H^2} + \|\partial_t^{q}w(0)\|_{H^2} \le C,
\end{split}
\end{equation}
then the IMEX-BDF scheme \eqref{BDF} is applied at any starting time $T_0\ge 0$. 

{\bf Case 2.} For $q=1,2,3,4$, if the initial data satisfies
\begin{equation}\label{thm_accu_0}
\|u(0)\|_{H^{q+2}}^2 + \|w(0)\|_{H^{q+2}}^2 \le C,
\end{equation}
then the IMEX-BDF scheme \eqref{BDF} is applied at the starting time $T_0\ge (q+2)\eps \log(1/\eps)$. 

The above treatment makes sure the initial layer is passed and the initial data is well prepared so that we have the following:

\begin{theorem}[Uniform accuracy of IMEX-BDF schemes]
\label{thm_accu}
Under the conditions of either {\bf Case 1} or {\bf Case 2}, the IMEX-BDF scheme \eqref{BDF} is uniformly $q$-th order accurate under the CFL condition $\Delta t \le \min\{c_{CFL}/N^2,1\}$ for any constant $c_{CFL}>0$, that is,
\begin{equation}
\|u(t_n)-U_N^n\|^2 + \|w(t_n)-W_N^n\|^2 \le C(\Delta t^{2q} + e_\text{proj}),
\end{equation}
for any $n$ such that $t_n=T_0+n\Delta t\leq T$, where $C$ is a constant independent of $\varepsilon$, $N$ and $\Delta t$, and 
\begin{equation}\label{eproj}
e_\text{proj}:=\sum_{i=0}^{q-1}\left(\|(I-\mathcal{P}_N)u(t_i)\|^2 + \Big(1+\frac{\Delta t}{\eps}\Big)\|(I-\mathcal{P}_N)w(t_i)\|^2\right).
\end{equation}
\end{theorem}

The projection error $e_\text{proj}$ is small if one assumes enough regularity of the initial data. To be precise, we have
\begin{corollary}
Under the conditions of Theorem \ref{thm_accu}, if one further assumes 
\begin{equation}\begin{split}
& \|u(T_0)\|_{H^{2q+1}}^2 + \|w(T_0)\|_{H^{2q+1}}^2 \le C, \\
& \|\partial_tu(T_0)\|_{H^{2q}}^2 + \|\partial_tw(T_0)\|_{H^{2q}}^2 \le C,
\end{split}\end{equation}
then there holds the error estimate
\begin{equation}
\|u(t_n)-U_N^n\|^2 + \|w(t_n)-W_N^n\|^2 \le C\left(\Delta t^{2q} + \frac{1}{N^{4q}}\right),
\end{equation}
for any $n$ such that $t_n=T_0+n\Delta t\leq T$, where $C$ is a constant independent of $\varepsilon$, $N$ and $\Delta t$.
\end{corollary}
\begin{proof}
By Theorem \ref{thm_accu}, it suffices to prove that $e_\text{proj} \le C/N^{4q}$. In fact, \eqref{energy} implies
\begin{equation}
\|u(t)\|_{H^{2q+1}}^2 + \|w(t)\|_{H^{2q+1}}^2 \le C,
\end{equation}
for all $t\ge T_0$, since $(\partial_x^{s} u,\partial_x^{s} w)$ satisfies the same system \eqref{eq1}. Therefore, using Lemma \ref{Fourier1}, 
\begin{equation}\begin{split}
\|(I-\mathcal{P}_N)u(t)\|^2& \leq \frac{1}{N^{4q+2}} \|u\|_{H^{2q+1}}^2 \leq \frac{C}{N^{4q+2}}.
\end{split}\end{equation}
Similar treatment for $w,\partial_x u, \partial_x w, \partial_t w$ gives
\begin{equation}\begin{split}
& \|(I-\mathcal{P}_N)w(t)\|^2\le \frac{C}{N^{4q+2}},\\
&  \|(I-\mathcal{P}_N)\partial_x u(t)\|^2+\|(I-\mathcal{P}_N)\partial_x w(t)\|^2+\|(I-\mathcal{P}_N)\partial_t w(t)\|^2 \le \frac{C}{N^{4q}}. \\
\end{split}\end{equation}
The second line above implies 
\begin{equation}
\|(I-\mathcal{P}_N)w(t)\|^2\le \eps^2\frac{C}{N^{4q}},
\end{equation}
since $w = -\eps(\partial_t w + \partial_x ((1-b^2)u-bw))$. Then the conclusion follows.
\end{proof}

\begin{proof}[Proof of Theorem \ref{thm_accu}]
We first prove the consistency of the IMEX-BDF schemes and then combine it with the stability to achieve the uniform accuracy.

If \eqref{thm_accu_cons} holds, then \eqref{energy} implies
\begin{equation}\label{energyqt}\begin{split}
& \|\partial_t^{q+1}u(t)\|_{H^1} + \|\partial_t^{q+1}w(t)\|_{H^1} \le C, \\
& \|\partial_t^{q}u(t)\|_{H^2} + \|\partial_t^{q}w(t)\|_{H^2} \le C,
\end{split}
\end{equation}
for all $t\ge 0$, since $(\partial_t^{s_1}\partial_x^{s_2} u,\partial_t^{s_1}\partial_x^{s_2} w)$ satisfies the same system \eqref{eq1}. If \eqref{thm_accu_0} holds, then Theorem \ref{thm_reg} implies \eqref{energyqt} for $t\ge T_0 \ge (q+2)\eps \log(1/\eps)$. Thus in both cases, using the Sobolev inequality,
\begin{equation}\begin{split}
& \|\partial_t^{q+1}u(t)\|_{L^\infty} + \|\partial_t^{q+1}w(t)\|_{L^\infty} \le \|\partial_t^{q+1}u(t)\|_{H^1} + \|\partial_t^{q+1}w(t)\|_{H^1} \le C, \\
& \|\partial_t^{q}\partial_xu(t)\|_{L^\infty} + \|\partial_t^{q}\partial_xw(t)\|_{L^\infty} \le \|\partial_t^{q}\partial_xu(t)\|_{H^1} + \|\partial_t^{q}\partial_xw(t)\|_{H^1} \le C.
\end{split}\end{equation}

For the exact solution $u(t,x)$ there holds the consistency error (pointwise in $x$):
\begin{equation}\begin{split}
&\big| \sum_{i=0}^q\alpha_i u^{n+i} - \beta\Delta t\partial_t u^{n+q} \big|\leq  C \Delta t^{q+1}   \max_{t\in[T_0,T]}  |\partial_t^{q+1} u|, \\
& \big| \Delta t\sum_{i=0}^{q-1}\gamma_i  \partial_xu^{n+i} -\beta \Delta t \partial_x u^{n+q} \big| \leq C \Delta t^{q+1} \max_{t\in[T_0,T]} \big|\partial_t^{q}\partial_x u\big| ,
\end{split}\end{equation}
where $u^n := u(t_n,x)$. Hence, 
\begin{equation}\label{cons}\begin{split}
&\big\| \sum_{i=0}^q\alpha_i u^{n+i} - \beta\Delta t\partial_t u^{n+q} \big\|\leq C \big\| \sum_{i=0}^q\alpha_i u^{n+i} - \beta\Delta t\partial_t u^{n+q} \big\|_{L^{\infty}} \leq C\Delta t^{q+1}, \\
&\big \| \Delta t\sum_{i=0}^{q-1}\gamma_i  \partial_xu^{n+i} -\beta \Delta t \partial_x u^{n+q} \big\| \leq  C\big\| \Delta t\sum_{i=0}^{q-1}\gamma_i  \partial_xu^{n+i} -\beta \Delta t \partial_x u^{n+q}\big \|_{L^{\infty}}\leq C \Delta t^{q+1}.
\end{split}\end{equation}
Similar estimates hold for the exact solution $w(t,x)$.

Let $S_U^n$ and $S_W^n$ be the truncation errors of the scheme, i.e., the remainders obtained by inserting the exact solutions $u(t,x)$ and $w(t,x)$ into (\ref{BDF}):
\begin{equation} \label{BDF1}
\left\{\begin{split}
& \sum_{i=0}^q\alpha_i u^{n+i} + \Delta t \sum_{i=0}^{q-1}\gamma_i \partial_x (b u^{n+i}+w^{n+i}) = S_U^n, \\
& \sum_{i=0}^q\alpha_i w^{n+i} + \Delta t \sum_{i=0}^{q-1}\gamma_i \partial_x ((1-b^2)u^{n+i}-bw^{n+i}) = -\frac{\beta\Delta t}{\eps}w^{n+q}+S_W^n.
\end{split}\right.\end{equation}
Then using the estimates in (\ref{cons}), we have
\begin{equation}
\|S^n_U\| + \|S^n_W\| \leq C \Delta t^{q+1}.
\end{equation}

Define the errors $\delta U^n = u^n-U_N^n$, $\delta W^n = w^n-W_N^n$; then subtracting (\ref{BDF}) from (\ref{BDF1}) gives
\begin{equation}\left\{\begin{split}
& \sum_{i=0}^q\alpha_i \delta U^{n+i} + \Delta t \sum_{i=0}^{q-1}\gamma_i \partial_x (b\delta U^{n+i}+\delta W^{n+i}) = S^n_U ,\\
& \sum_{i=0}^q\alpha_i \delta W^{n+i} + \Delta t \sum_{i=0}^{q-1}\gamma_i \partial_x ((1-b^2)\delta U^{n+i}-b\delta W^{n+i}) = -\frac{\beta\Delta t}{\eps}\delta W^{n+q} +S^n_W.
\end{split}\right.\end{equation}

Therefore, $(\delta U,\delta W)$ satisfies the same scheme \eqref{BDF} up to the source terms. An argument similar to the proof of Theorem \ref{thm_stab} gives
\begin{equation}\label{est_delta}\begin{split}
& \left((1-b^2)G^{n+1}_{\delta U}+G^{n+1}_{\delta W}+\frac{\beta\Delta t}{\eps}A^{n+1}_{\delta W}\right)-\left((1-b^2)G^{n}_{\delta U}+G^{n}_{\delta W}+\frac{\beta\Delta t}{\eps}A^{n}_{\delta W}\right) \\
\le & C\Delta t\sum_{i=0}^{q-1}\left(\left\| \delta U^{n+i}\right\|^2 + \left\| \delta W^{n+i}\right\|^2\right) + \kappa \Delta t \left(\|\delta U^{n+q}\|^2+\|\delta W^{n+q}\|^2\right) + \frac{C}{\kappa }\Delta t^{2q+1},
\end{split}\end{equation}
where the term involving $S^n_U$ is estimated by ($S^n_W$ is estimated similarly)
\begin{equation}
 \left|\Re\int \left(\delta \bar{U}^{n+q}-\sum_{i=1}^{q-1}\eta_i \delta\bar{U}^{n+i}\right)S_U^n \rd{x}\right|
\le \kappa \Delta t \|\delta U^{n+q}\|^2 + C\Delta t\sum_{i=1}^{q-1}\|\delta U^{n+i}\|^2 + \frac{C}{\kappa }\Delta t^{2q+1},
\end{equation}
with $\kappa >0$ small, to be chosen. 

Similarly as in Theorem \ref{thm_stab}, define
\begin{equation}
E^n=(1-b^2)G^{n}_{\delta U}+G^{n}_{\delta W}+\frac{\beta\Delta t}{\eps}A^{n}_{\delta W};
\end{equation}
then by the postive-definiteness of $G$ and semi-positive-definiteness of $A$, we have
\begin{equation}\begin{split}
\sum_{i=0}^{q-1}\left(\left\| \delta U^{n+i}\right\|^2 + \left\| \delta W^{n+i}\right\|^2\right) \le \frac{1}{(1-b^2)\lambda_{\min}} E^n,
\end{split}\end{equation}
and
\begin{equation}\begin{split}
 \|\delta U^{n+q}\|^2+\|\delta W^{n+q}\|^2 \le \frac{1}{(1-b^2)\lambda_{\min}} E^{n+1}.
\end{split}\end{equation}
Using these in \eqref{est_delta} gives
\begin{equation}\begin{split}
\left(1-\frac{\kappa\Delta t }{(1-b^2)\lambda_{\min}} \right) E^{n+1}\le (1+C\Delta t)E^n+ \frac{C}{\kappa}\Delta t^{2q+1}.
\end{split}\end{equation}
Recall that $\Delta t\le 1$ by assumption. By choosing $\kappa  = \frac{(1-b^2)\lambda_{\min}}{2}$, we have
\begin{equation}
\frac{1}{1-\frac{\kappa \Delta t}{(1-b^2)\lambda_{\min}}} = \frac{1}{1-\frac{\Delta t}{2}} \le 1+ \Delta t.
\end{equation}
Hence
\begin{equation}
E^{n+1}\le (1+C\Delta t)E^n+ C\Delta t^{2q+1}.
\end{equation}
Then using Gronwall's inequality, we obtain
\begin{equation}
E^n\leq C (E_0+\Delta t^{2q}).
\end{equation}
Finally the conclusion follows by noticing that 
\begin{equation}
\|\delta U^n\|^2+\|\delta W^n\|^2\leq CE^n, \quad E^0\leq  C\sum_{i=0}^{q-1}\left(\|\delta U_N^{i}\|^2+\Big(1+\frac{\Delta t}{\eps}\Big)\|\delta W_N^{i}\|^2\right)=C e_\text{proj}.
\end{equation}
\end{proof}

\section{Extension to the variable coefficient case}
\label{sec:variable}

In this section we show that our results on the uniform accuracy of IMEX-BDF schemes can be extended to linear hyperbolic relaxation systems with variable coefficients:
\begin{equation}\label{eq2_0}\left\{\begin{split}
& \partial_t u + \partial_x v = 0, \\
& \partial_t v + \partial_x u = \frac{\sigma(x)}{\eps}(b(x)u-v),
\end{split}\right.\end{equation}
where $b(x)$ and $\sigma(x)$ are smooth, satisfying
\begin{equation}
|b(x)|\le b_1<1,\quad  0<\sigma_0\le \sigma(x) \le \sigma_1, \quad x\in[0,2\pi].
\end{equation}
In particular, the rationale of considering $\sigma(x)$ is that it resembles the collision frequency in the kinetic equation.

For simplicity, we will only show the analogs of Theorem \ref{thm_reg} with $s=1$ and Theorem \ref{thm_stab} with $q=1$. Higher order case can be obtained in similar ways.

Introducing the new variable $w = v-b(x)u$, one can rewrite \eqref{eq2_0} as
\begin{equation}\label{eq2}\left\{\begin{split}
& \partial_t u + \partial_x (b(x)u+w) = 0, \\
& \partial_t w + \partial_x ((1-b(x)^2)u) + b(x)\partial_x b(x) u - b (x)\partial_x w= -\frac{\sigma(x)}{\eps}w.
\end{split}\right.\end{equation}
Multiplying the first equation by $(1-b^2)u$ and integrating in $x$, we get
\begin{equation}
\partial_t \frac{1}{2}\|u\|_b^2 + \int (1-b^2)u\partial_x w \rd{x} + \int \frac{1}{2}(1+b^2)(\partial_xb) u^2\rd{x} = 0 ,\quad \|u\|_b^2 = \int (1-b^2)u^2\rd{x},
\end{equation}
using the notation $\|u\|_b$ for a weighted $L^2$ norm, being equivalent to $\|u\|$ since $|b(x)|\le b_1<1$. Multiplying the second equation of \eqref{eq2} by $w$ and integrating in $x$, we get
\begin{equation}
\partial_t \frac{1}{2}\|w\|^2 + \int w\partial_x ((1-b^2)u) \rd{x} + \int (b\partial_xb) u w \rd{x} + \frac{1}{2}\int (\partial_xb) w^2\rd{x} = -\frac{1}{\eps}\int \sigma w^2\rd{x}.
\end{equation}
Adding the above two relations gives the energy estimate
\begin{equation}\label{energyb}
\partial_t \frac{1}{2}(\|u\|_b^2+\|w\|^2) \le -\frac{\sigma_0}{\eps}\|w\|^2 + C(\|u\|^2+\|w\|^2),
\end{equation}
since $b(x)$ is a smooth function.

We first state the regularity result, analogous to Theorem \ref{thm_reg} with $s=1$:
\begin{theorem}\label{thm_regb}
Assume 
\begin{equation}\label{thm_regb_0}
\|u(0)\|_{H^1}^2 + \|w(0)\|_{H^1}^2 =: E_0 < \infty.
\end{equation}
Then  the solution to \eqref{eq2} satisfies
\begin{equation}\label{thm_regb_1}
\|\partial_t^{r_1}\partial_x^{r_2} u(t)\|^2 + \|\partial_t^{r_1}\partial_x^{r_2} w(t)\|^2 \le CE_0e^{Ct},\quad r_1+r_2 \le 1,
\end{equation}
and
\begin{equation}\label{thm_regb_2}
\|w(t)\|^2 \le CE_0\eps^2e^{Ct},
\end{equation}
for all $t\ge \frac{\eps \log(1/\eps)}{\sigma_0}$. 
\end{theorem}
\begin{proof}
$(\partial_x u, \partial_x w)$ satisfies
\begin{equation}\left\{\begin{split}
& \partial_t \partial_xu + \partial_x (b\partial_xu+\partial_xw) = -\partial_x(\partial_xb u), \\
& \partial_t \partial_xw + \partial_x ((1-b^2)\partial_xu) + b\partial_xb \partial_xu - b\partial_{xx}w = -\frac{\sigma}{\eps}\partial_xw +  2\partial_x(b\partial_xb u) - \partial_x(b\partial_xb)u +\partial_xb\partial_xw -\frac{\partial_x\sigma}{\eps}w.
\end{split}\right.\end{equation}
Therefore, similar energy estimate as above gives
\begin{equation}\label{est_b}
\partial_t \frac{1}{2}\left(\|\partial_xu\|_b^2+\|\partial_xw\|^2\right) \le -\frac{\sigma_0}{2\eps}\|\partial_xw\|^2 +C(\|\partial_xu\|^2+\|\partial_xw\|^2 + \|u\|^2+\|w\|^2) + \frac{C_1}{\eps} \|w\|^2,
\end{equation}
where the term $\int \partial_x w (-\frac{\partial_x\sigma}{\eps}w)\rd{x}$ is handled by 
\begin{equation}
\left|\int \partial_x w \left (-\frac{\partial_x\sigma}{\eps}w\right)\rd{x}\right| \le \frac{\|\partial_x\sigma\|_{L^\infty}}{\eps}\int |w\partial_xw|\rd{x} \le \frac{\sigma_0}{2\eps}\|\partial_xw\|^2 + \frac{\|\partial_x\sigma\|_{L^\infty}^2}{2\sigma_0 }\frac{1}{\eps}\|w\|^2.
\end{equation}
Multiplying \eqref{energyb} by $C_1/\sigma_0$ and adding to \eqref{est_b}, we can absorb the last term $\frac{C_1}{\eps} \|w\|^2$ and get
\begin{equation}\begin{split}
& \partial_t \left(\frac{1}{2}(\|\partial_xu\|_b^2+\|\partial_xw\|^2) + \frac{C_1}{2\sigma_0}(\|u\|_b^2+\|w\|^2)\right) \\
 \le & -\frac{\sigma_0}{2\eps}\|\partial_xw\|^2 +C\left(\|\partial_xu\|^2+\|\partial_xw\|^2 + \|u\|^2+\|w\|^2\right) ,
\end{split}\end{equation}
and then it follows that
\begin{equation}\label{energyb2}
\|\partial_x u(t)\|^2 + \|\partial_x w(t)\|^2 \le CE_0e^{Ct},
\end{equation}
for all $t\ge 0$. 

$(\partial_t u,\partial_t w)$ satisfies the same system \eqref{eq2}, and thus has the energy estimate
\begin{equation} \label{energyb3}
\begin{split}
\partial_t \frac{1}{2}\left(\|\partial_tu\|_b^2+\|\partial_tw\|^2\right) \le & -\frac{\sigma_0}{\eps}\|\partial_tw\|^2 + C\left(\|\partial_tu\|^2+\|\partial_tw\|^2\right) \\
= & -\frac{\sigma_0}{\eps}\left(\|\partial_tu\|_b^2+\|\partial_tw\|^2\right) + C\left(\|\partial_tu\|^2+\|\partial_tw\|^2\right) + \frac{\sigma_0}{\eps}\|\partial_tu\|_b^2 \\
\le & -\Big(\frac{\sigma_0}{\eps}-C_1\Big)\left(\|\partial_tu\|_b^2+\|\partial_tw\|^2\right)  + \frac{\sigma_0}{\eps}\|\partial_tu\|_b^2,
\end{split}\end{equation}
for some $C_1>0$. Notice that $\partial_t u = -\partial_x (b(x)u+w)$. Thus by \eqref{energyb2},
\begin{equation}
\|\partial_t u(t)\|^2 \le CE_0e^{Ct}.
\end{equation}
Similarly
\begin{equation}
\|\partial_t w(t)\|^2 \le \frac{1}{\eps^2}CE_0e^{Ct}.
\end{equation}
Then by \eqref{energyb3}, 
\begin{equation}\begin{split}
& \frac{1}{2}\left(\|\partial_t  u(t)\|_b^2+\|\partial_t  w(t)\|^2\right)\\
\le & e^{-2(\frac{\sigma_0}{\eps}-C_1) t}\frac{1}{2}\left(\|\partial_t  u(0)\|_b^2+\|\partial_t  w(0)\|^2\right)  + \int_0^t e^{-2(\frac{\sigma_0}{\eps}-C_1)(t-\tau)}\frac{\sigma_0}{\eps}\|\partial_\tau u(\tau)\|_b^2\rd{\tau} \\
\le & e^{2C_1t}\left(e^{-2\sigma_0t/\eps}\frac{1}{2}\left(\|\partial_t  u(0)\|_b^2+\|\partial_t  w(0)\|^2\right)  + \int_0^t e^{-2\sigma_0(t-\tau)/\eps}\frac{\sigma_0}{\eps}\|\partial_\tau u(\tau)\|_b^2\rd{\tau}\right) \\
\le & \left(e^{-2\sigma_0t/\eps}\frac{1}{\eps^2} +1\right)CE_0e^{(C+2C_1)t}, \\
\end{split}\end{equation}
and it follows that 
\begin{equation}
\|\partial_t w(t)\|^2 \le CE_0e^{Ct},
\end{equation}
for $t\ge \frac{\eps \log(1/\eps)}{\sigma_0}$. \eqref{thm_regb_2} follows from \eqref{thm_regb_1} using the relation $w = -\frac{\eps}{\sigma}(\partial_t w + \partial_x ((1-b^2)u) + b\partial_x b u - b \partial_x w)$.
\end{proof}

To apply the spectral discretization in space, we first need to rewrite \eqref{eq2} using the variables $(\tu,w)$, $\tu = u\sqrt{1-b^2}$:\footnote{The purpose of this reformulation is to make the Lyapunov functional $\|u\|^2_b+\|w\|^2$ into the pure $L^2$ norm $\|\tu\|^2+\|w\|^2$, so that no further error is produced by the Galerkin projection when conducting energy estimates.}
\begin{equation}\label{eq2_1}\left\{\begin{split}
& \partial_t \tu  + \partial_x (b\tu) + \frac{b^2\partial_x b}{1-b^2}\tu + \sqrt{1-b^2}\partial_x w = 0, \\
& \partial_t w + \partial_x (\sqrt{1-b^2}\tu) + \frac{b\partial_x b}{\sqrt{1-b^2}} \tu - b\partial_x w = -\frac{\sigma}{\eps}w.
\end{split}\right.\end{equation}
Based on this reformulation, the first order IMEX-BDF (i.e., forward-backward Euler) scheme together with the Fourier-Galerkin spectral method reads
\begin{equation}\label{BDFb}\left\{\begin{split}
& \tU_N^{n+1}-\tU_N^n +  \Delta t \mathcal{P}_N\left( \partial_x (b\tU_N^n) + \frac{b^2\partial_x b}{1-b^2}\tU_N^n + \sqrt{1-b^2}\partial_x W_N^n \right) = 0, \\
& W_N^{n+1}-W_N^n + \Delta t \mathcal{P}_N \left[ \partial_x \left(\sqrt{1-b^2}\tU_N^{n}\right) + \frac{b\partial_x b}{\sqrt{1-b^2}} \tU_N^n - b\partial_xW_N^n \right]= -\frac{\Delta t}{\eps}\mathcal{P}_N(\sigma W_N^{n+1}).
\end{split}\right.\end{equation}
We have the following stability result:

\begin{theorem}\label{thm_stabb}
Under the CFL condition $\Delta t \le \min\{c_{CFL}/N^2,1\}$ for any $c_{CFL}>0$, \eqref{BDFb} is uniformly stable, in the sense that
\begin{equation}
\|\tU_N^n\|^2 + \|W_N^n\|^2 \le C\left(\|\tU_N^0\|^2+\|W_N^0\|^2\right),
\end{equation}
for any $n$ such that $t_n=T_0+n\Delta t\leq T$, where $C$ is a constant independent of $\eps$, $\Delta t$ and $N$.
\end{theorem}

\begin{proof}
Multiplying the first equation of \eqref{BDFb} by $\bar{\tU}_N^{n+1}$, integrating in $x$ and taking the real part gives
\begin{equation}\label{GUb}\begin{split}
0 = & \frac{1}{2}\|\tU_N^{n+1}\|^2 - \frac{1}{2}\|\tU_N^{n}\|^2 + \frac{1}{2}\|\tU_N^{n+1}-\tU_N^n\|^2 + \Delta t \Re\int \bar{\tU}_N^{n+1}\partial_x (b\tU_N^n) \rd{x} \\
& + \Delta t \Re\int \bar{\tU}_N^{n+1}\frac{b^2\partial_x b}{1-b^2}\tU_N^n \rd{x} + \Delta t \Re\int \bar{\tU}_N^{n+1}\sqrt{1-b^2}\partial_x W_N^n  \rd{x} \\
= & \frac{1}{2}\|\tU_N^{n+1}\|^2 - \frac{1}{2}\|\tU_N^{n}\|^2 + \frac{1}{2}\|\tU_N^{n+1}-\tU_N^n\|^2 + \Delta t \Re\int \frac{b^2\partial_x b}{1-b^2}\bar{\tU}_N^{n+1}\tU_N^n \rd{x}\\
& + \Delta t \Re\int (\bar{\tU}_N^{n+1}-\bar{\tU}_N^n)\partial_x (b\tU_N^n) \rd{x}+ \Delta t \Re\int \bar{\tU}_N^n\partial_x (b\tU_N^n) \rd{x} \\
&  + \Delta t \Re\int (\bar{\tU}_N^{n+1}-\bar{\tU}_N^n)\sqrt{1-b^2}\partial_x W_N^n  \rd{x} + \Delta t \Re\int \bar{\tU}_N^{n}\sqrt{1-b^2}\partial_x W_N^n  \rd{x}, \\
\end{split}\end{equation}
where we used the fact that $\mathcal{P}_N$ is self-adjoint (Lemma \ref{adjoint}) and $\mathcal{P}_N(\tU_N^{n+1})=\tU_N^{n+1}$.

Multiplying the second equation of \eqref{BDFb} by $\bar{W}_N^{n+1}$, integrating in $x$ and taking the real part gives
\begin{equation}\label{GWb}\begin{split}
& -\frac{\Delta t}{\eps}\int \sigma |W_N^{n+1}|^2 \rd{x} \\
 = & \frac{1}{2}\|W_N^{n+1}\|^2 - \frac{1}{2}\|W_N^{n}\|^2 + \frac{1}{2}\|W_N^{n+1}-W_N^n\|^2 \\
 & + \Delta t \Re\int \bar{W}_N^{n+1}\partial_x(\sqrt{1-b^2}\tU_N^{n})  \rd{x} + \Delta t\Re\int \bar{W}_N^{n+1}\frac{b\partial_x b}{\sqrt{1-b^2}} \tU_N^n\rd{x} - \Delta t \Re\int \bar{W}_N^{n+1}b\partial_x(W_N^n)  \rd{x} \\
= & \frac{1}{2}\|W_N^{n+1}\|^2 - \frac{1}{2}\|W_N^{n}\|^2 + \frac{1}{2}\|W_N^{n+1}-W_N^n\|^2+ \Delta t\Re\int \frac{b\partial_x b}{\sqrt{1-b^2}} \bar{W}_N^{n+1}\tU_N^n\rd{x}  \\
& - \Delta t \Re\int (\bar{W}_N^{n+1}-\bar{W}_N^n)b\partial_x(W_N^n)  \rd{x} - \Delta t \Re\int \bar{W}_N^{n}b\partial_x(W_N^n)  \rd{x} \\
& + \Delta t \Re\int (\bar{W}_N^{n+1}-\bar{W}_N^n)\partial_x(\sqrt{1-b^2}\tU_N^{n})  \rd{x}+ \Delta t \Re\int \bar{W}_N^{n}\partial_x(\sqrt{1-b^2}\tU_N^{n})  \rd{x}.\\
\end{split}\end{equation}

We notice
\begin{equation}
\Re\int \bar{\tU}_N^{n}\sqrt{1-b^2}\partial_x W_N^n  \rd{x}  + \Re\int \bar{W}_N^{n}\partial_x(\sqrt{1-b^2}\tU_N^{n})  \rd{x} = 0,
\end{equation}
and estimate the terms
\begin{equation}\begin{split}
& \Delta t \left|\Re\int \frac{b^2\partial_x b}{1-b^2}\bar{\tU}_N^{n+1}\tU_N^n \rd{x}\right| \le \kappa \Delta t\|\tU_N^{n+1}\|^2+ \frac{C}{\kappa }\Delta t\|\tU_N^n\|^2, \\
& \Delta t \left|\Re\int (\bar{\tU}_N^{n+1}-\bar{\tU}_N^n)\partial_x (b\tU_N^n) \rd{x}\right| \le \frac{1}{4}\|\tU_N^{n+1}-\tU_N^n\|^2 + C\Delta t^2\left(\|\tU_N^n\|^2+\|\partial_x\tU_N^n\|^2\right), \\
& \Delta t \left|\Re\int (\bar{\tU}_N^{n+1}-\bar{\tU}_N^n)\sqrt{1-b^2}\partial_x W_N^n\rd{x}\right| \le \frac{1}{4}\|\tU_N^{n+1}-\tU_N^n\|^2 + C\Delta t^2\|\partial_xW_N^n\|^2, \\
& \Delta t \left|\Re\int \bar{\tU}_N^n\partial_x (b\tU_N^n) \rd{x}\right| = \Delta t \left|\Re\int b\partial_x\bar{\tU}_N^n\tU_N^n \rd{x}\right| = \Delta t \left|\int b\frac{1}{2}\partial_x|\tU_N^n|^2 \rd{x}\right| \\
& = \Delta t \left|\int \frac{1}{2}(\partial_x b)|\tU_N^n|^2 \rd{x}\right| \le C\Delta t\|\tU_N^n\|^2,
\end{split}\end{equation}
with $\kappa >0$ small, to be chosen, and similar for those terms in (\ref{GWb}). Therefore, adding \eqref{GUb} and \eqref{GWb} yields
\begin{equation}\begin{split}
& \frac{1}{2}\left(\|\tU_N^{n+1}\|^2+\|W_N^{n+1}\|^2\right) - \frac{1}{2}\left(\|\tU_N^{n}\|^2+\|W_N^{n}\|^2\right) \\
\le & \kappa \Delta t \left(\|\tU_N^{n+1}\|^2+\|W_N^{n+1}\|^2\right) + C\left(1+\frac{1}{\kappa }\right)\Delta t\left(\|\tU_N^n\|^2+\|W_N^n\|^2\right) + C\Delta t^2\left(\|\partial_x\tU_N^n\|^2+\|\partial_xW_N^n\|^2\right).
\end{split}\end{equation}

Under the CFL condition, one has (similarly for $W_N^n$)
\begin{equation}
\|\partial_x \tU_N^n\| \le N\|\tU_N^n\| \le \sqrt{\frac{c_{CFL}}{\Delta t}}\|U_N^n\|.
\end{equation}
Therefore
\begin{equation}\begin{split}
& \frac{1}{2}\left(\|\tU_N^{n+1}\|^2+\|W_N^{n+1}\|^2\right) - \frac{1}{2}\left(\|\tU_N^{n}\|^2+\|W_N^{n}\|^2\right) \\
\le & \kappa \Delta t \left(\|\tU_N^{n+1}\|^2+\|W_N^{n+1}\|^2\right) + C\left(1+\frac{1}{\kappa }\right)\Delta t\left(\|\tU_N^n\|^2+\|W_N^n\|^2\right) ,
\end{split}\end{equation}
which implies
\begin{equation}\begin{split}
& \|\tU_N^{n+1}\|^2+\|W_N^{n+1}\|^2 \le \frac{1+C(1+\frac{1}{\kappa })\Delta t}{1-2\kappa \Delta t}\left(\|\tU_N^{n}\|^2+\|W_N^{n}\|^2\right). \\
\end{split}\end{equation}
Recall that $\Delta t\leq 1$. By choosing $\kappa  = \frac{1}{4}$, we have
\begin{equation}
\frac{1}{1-2\kappa \Delta t} \le 1+ \Delta t.
\end{equation}
Then
\begin{equation}\begin{split}
& \|\tU_N^{n+1}\|^2+\|W_N^{n+1}\|^2 \le (1+C\Delta t)\left(\|\tU_N^{n}\|^2+\|W_N^{n}\|^2\right), \\
\end{split}\end{equation}
and the conclusion follows.
\end{proof}

\section{Numerical tests}
\label{sec:num}


%
%

In this section, we numerically test the accuracy of the IMEX-BDF schemes applied to several stiff hyperbolic and kinetic equations. We first verify our theoretical results in Section~\ref{subsec:accuracy} by considering the linear system (\ref{linearhyp}) for which the exact solution is available. We then consider a nonlinear hyperbolic relaxation system and the kinetic BGK equation. See Appendix~\ref{appendix:A} for a brief introduction of the BGK equation and its IMEX-BDF discretization.


\subsection{A linear stiff hyperbolic relaxation system}

Consider the linear system (\ref{linearhyp}) with $b=0.6$ on $x\in[0,1]$ with periodic boundary condition and initial condition
\begin{equation}
u(0,x) = e^{\sin 2\pi x},\quad v(0,x) = b e^{\sin 2\pi x}.
\end{equation}
We adopt the Fourier-Galerkin spectral method for spatial discretization with modes $|k|\leq N$, and fix $N=40$. The exact solution to this problem can be computed analytically using the Fourier transform in $x$. To avoid the initial layer or prepare the initial data satisfying the conditions of Theorem~\ref{thm_accu} (in particular, {\bf Case 2}), we start the computation from time $T_0=1$. The starting values at $T_0+i\Delta t$, $i=0,\dots,q-1$, are taken from the exact solution. We compute the solution to time $T=2$ and record the error as $\|U-u\|_{L^2} + \|V-v\|_{L^2}$.

Figure \ref{fig1} shows the results of IMEX-BDF schemes of order $q=2,3,4$, and various values of $\Delta t$ and $\eps$. In all the subfigures except the last one, each curve represents the error for a fixed $\Delta t$ with $\eps$ ranging from $1e-7$ to $1$. When taking the maximal $L^2$ error among all the tested values of $\eps$ for a fixed $\Delta t$ (see bottom right subfigure), the uniform $q$-th order accuracy is clearly achieved for $q=2,3,4$. This is in perfect agreement with our theoretical results.

\begin{figure}[htp]
\begin{center}
	\includegraphics[width=0.49\textwidth]{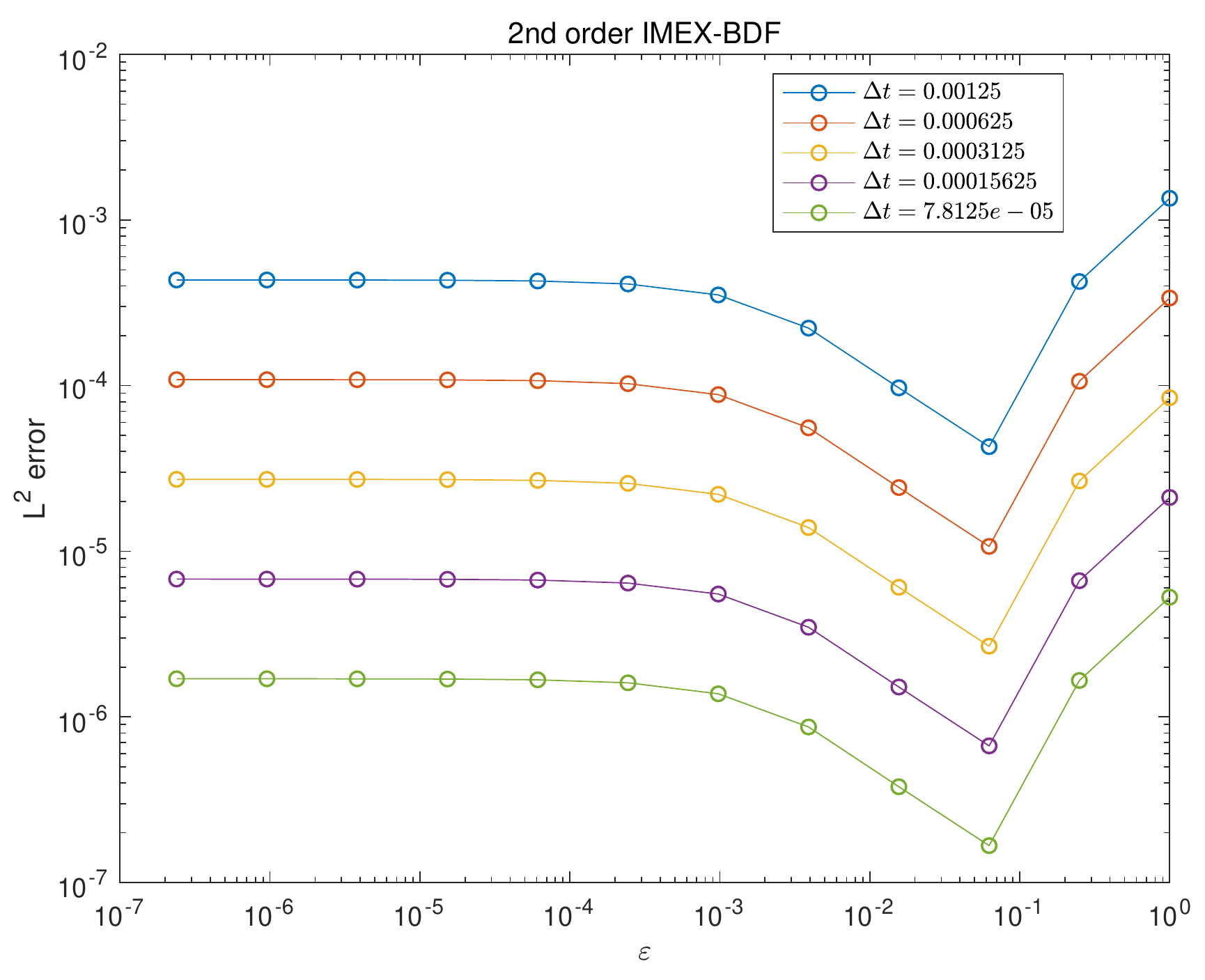}
	\includegraphics[width=0.49\textwidth]{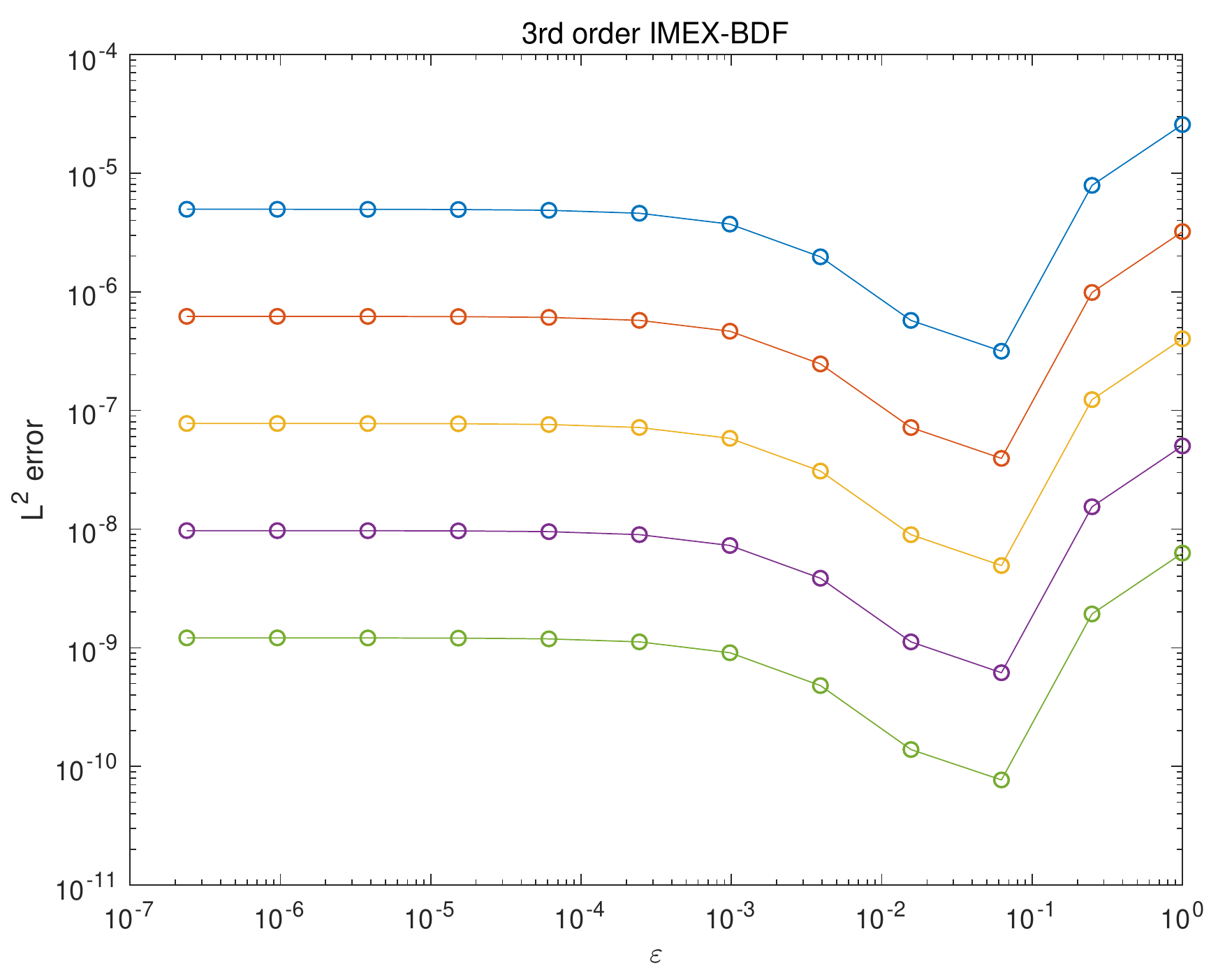}
	\includegraphics[width=0.49\textwidth]{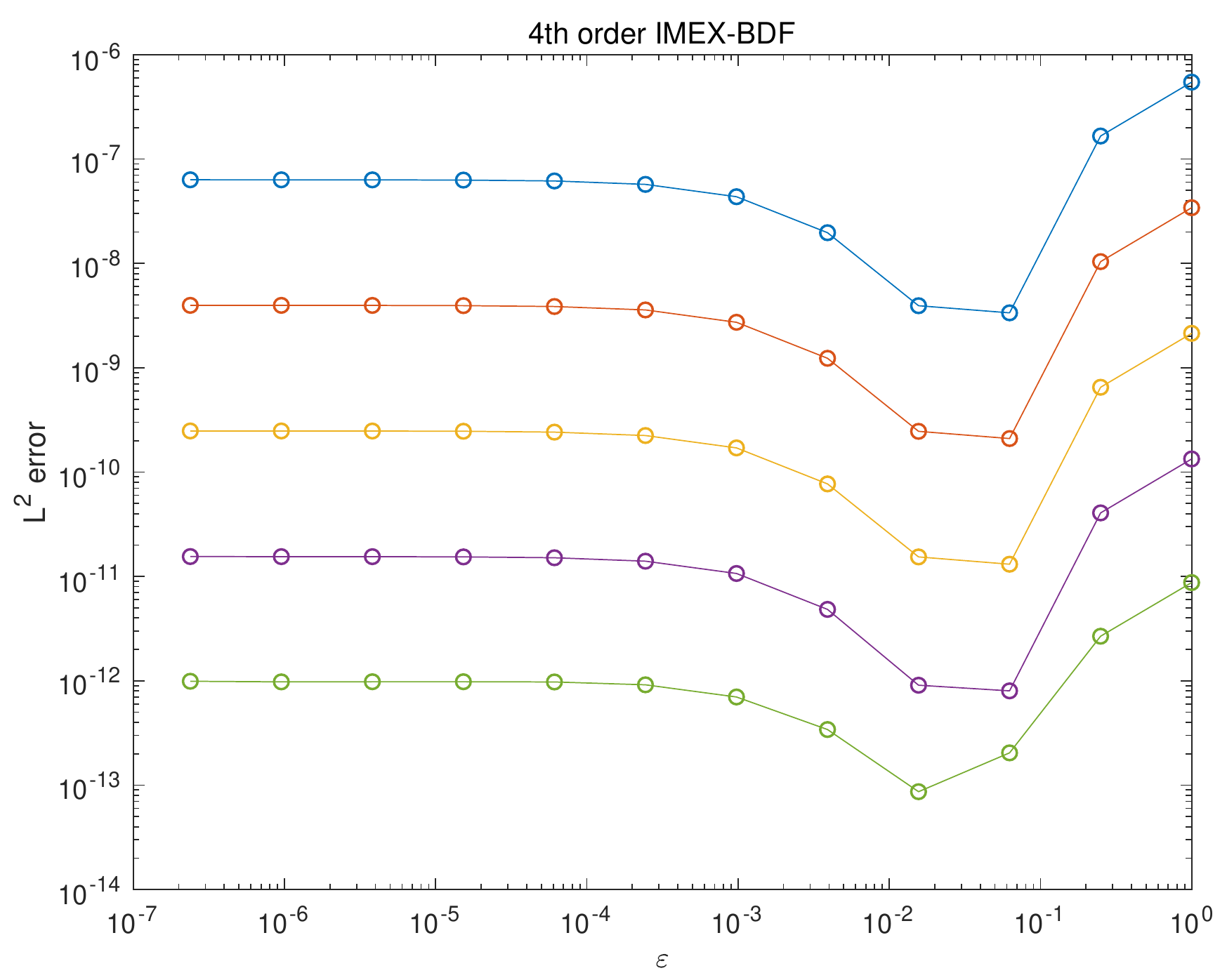}
	\includegraphics[width=0.49\textwidth]{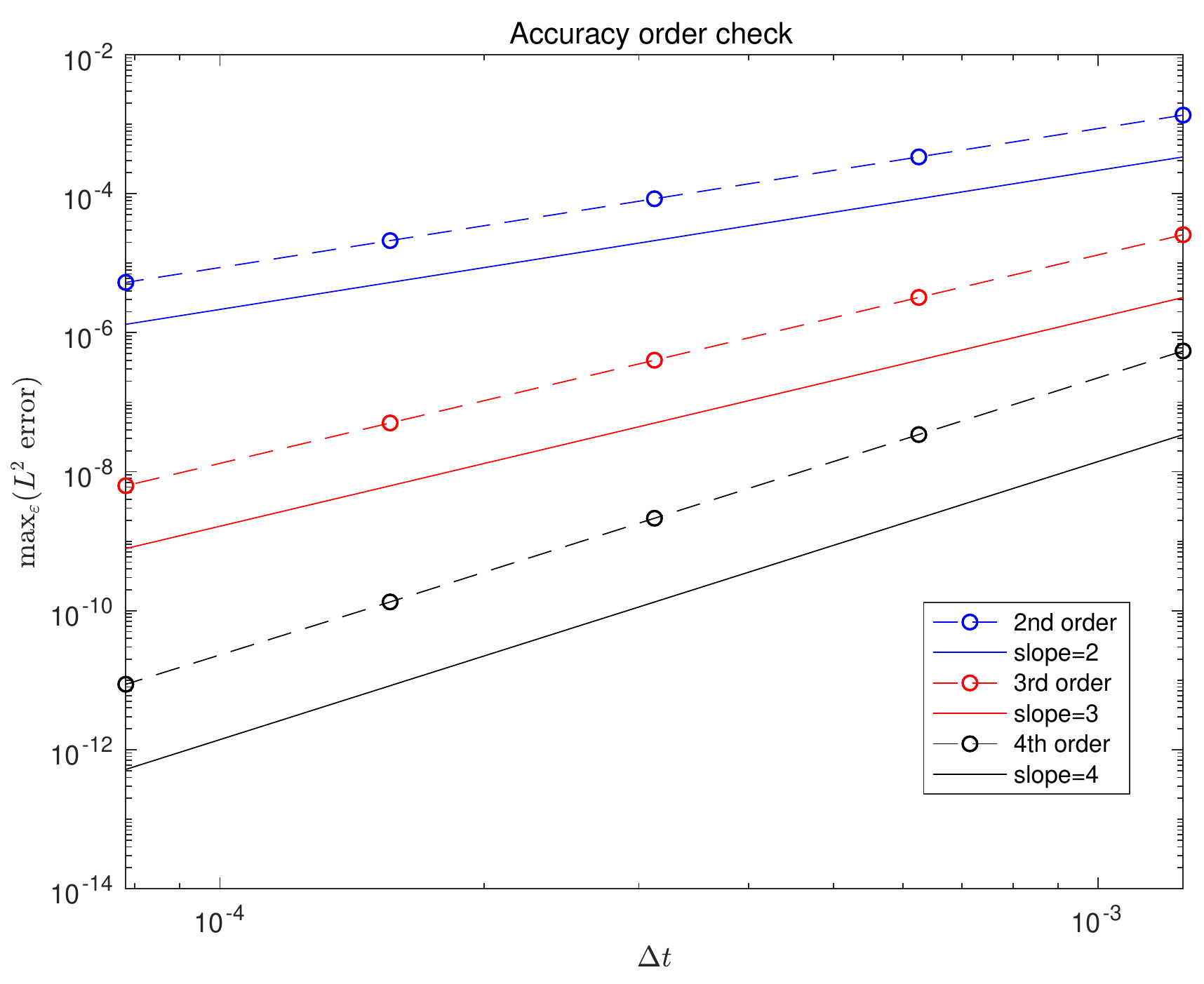}
	\caption{Linear stiff hyperbolic relaxation system. $L^2$ error of the solutions computed by IMEX-BDF schemes. Top left, top right and bottom left figures: second/third/fourth order IMEX-BDF schemes, respectively. In these three subfigures, horizontal axis is $\eps$ ranging from $1e-7$ to $1$, and different curves represent different values of $\Delta t$, as shown in the top left figure. Bottom right figure is obtained as follows: for each scheme, take the maximal $L^2$ error among all values of $\eps$ for a fixed $\Delta t$.}
	\label{fig1}
\end{center}	
\end{figure}

\subsection{A nonlinear stiff hyperbolic relaxation system}

We now consider the following nonlinear hyperbolic relaxation system
\begin{equation}\label{eqnu2}\left\{\begin{split}
& \partial_t u + \partial_x v = 0, \\
& \partial_t v + \partial_x u = \frac{1}{\eps}(bu^2-v),
\end{split}\right.\end{equation}
with $b=0.2$ on $x\in [0,1]$ with periodic boundary condition and initial condition
\begin{equation} 
u(0,x) = \frac{1}{2}e^{\sin2\pi x}=:u_0(x).
\end{equation}
Since the limit of (\ref{eqnu2}) is the Burgers equation which may develop shocks, we discretize in space by a fifth order finite volume WENO scheme \cite{Shu98}. We apply the second and third order IMEX-BDF schemes for time discretization. For the second order scheme, we choose the time step as $\Delta t = \frac{1}{4}\Delta x$ and the initial data for $v$ as
\begin{equation}\label{ini1}
v(0,x) = bu_0^2,
\end{equation}
which is consistent up to $O(1)$. For the third order scheme, we choose the time step as $\Delta t = \frac{1}{3}\Delta x$ and the initial data for $v$ as
\begin{equation}
v(0,x) = b u_0^2 - \eps(1-4b^2u_0^2)\partial_xu_0,
\end{equation}
which is consistent up to $O(\eps)$. The starting values at $i\Delta t$, $i=0,\dots,q-1$, are prepared using ARS(4,4,3) with a much smaller time step $\delta t = \Delta t/500$. We compute the solution to time $T=0.2$ and estimate the error of the solutions $U_{\Delta t, \Delta x}$, $V_{\Delta t, \Delta x}$ as $\|U_{\Delta t, \Delta x}-U_{\Delta t/2,\Delta x/2}\|_{L^2}+\|V_{\Delta t, \Delta x}-V_{\Delta t/2,\Delta x/2}\|_{L^2}$.

The results are shown in Figure \ref{fig2}. The uniform $q$-th order accuracy can be observed, similar to the previous subsection, except for the third order scheme for which one can see a slightly higher convergence rate for large $\Delta t$ due to the error from spatial discretization.

\begin{figure}[htp]
\begin{center}
	\includegraphics[width=0.49\textwidth]{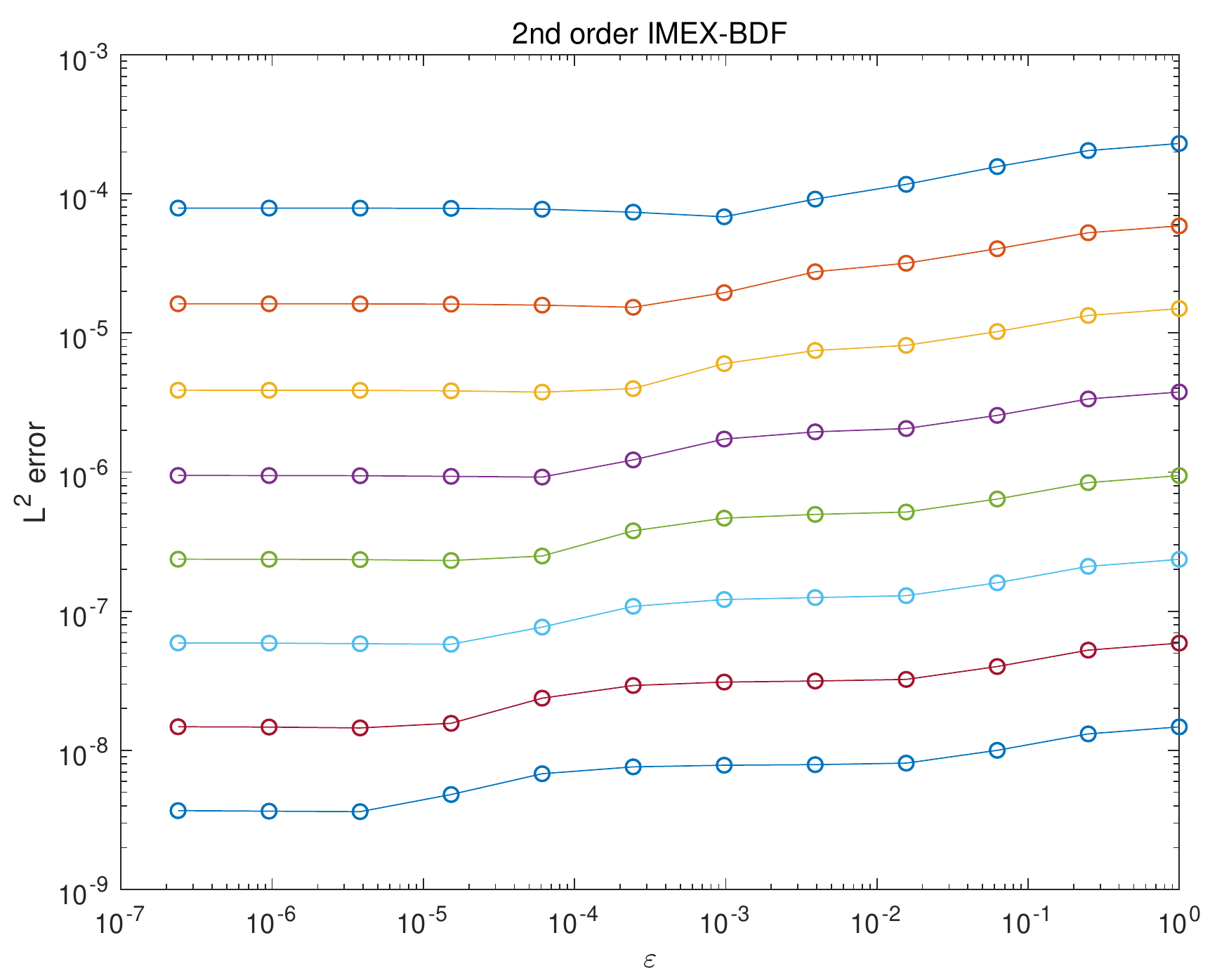}
	\includegraphics[width=0.49\textwidth]{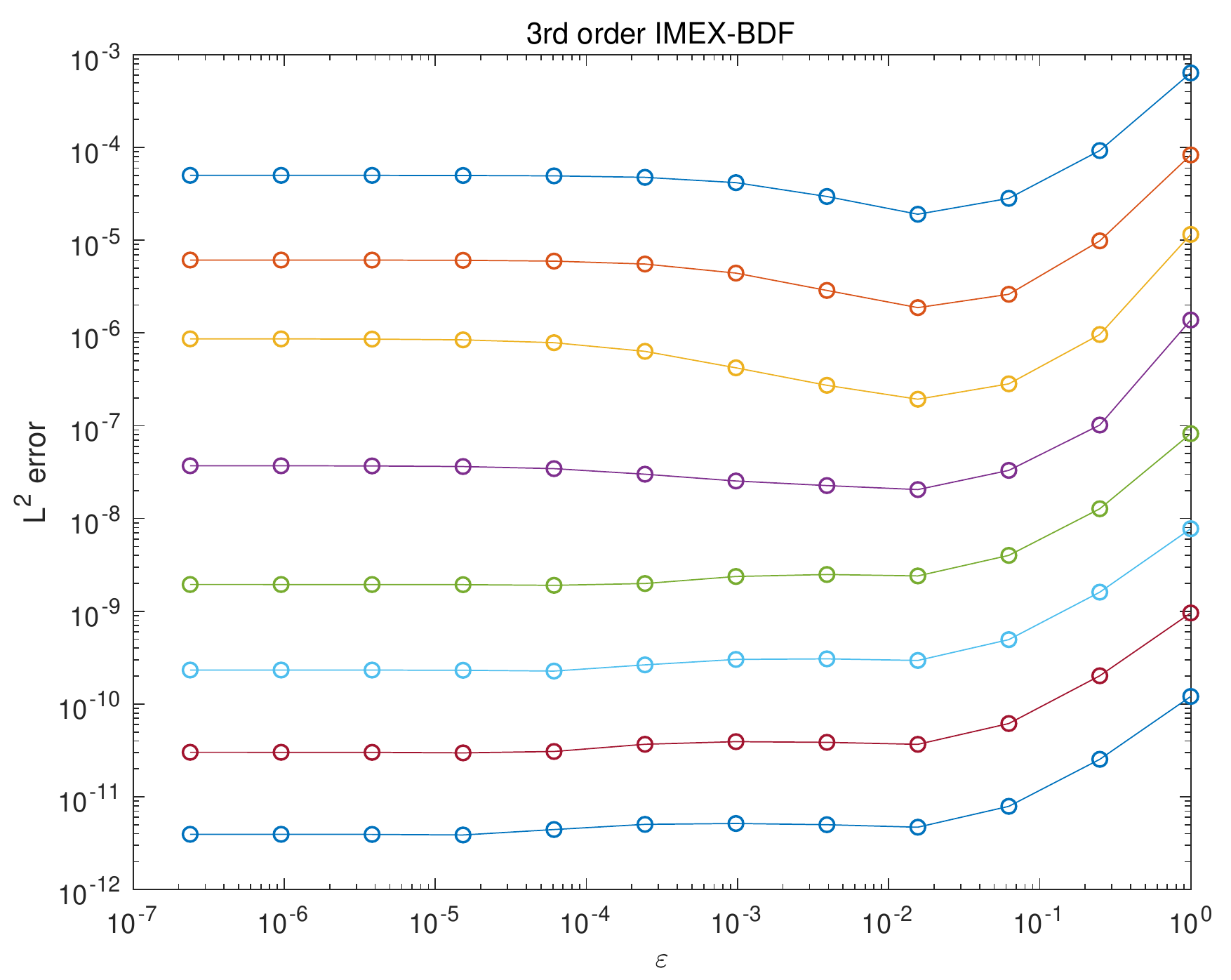}
	\includegraphics[width=0.49\textwidth]{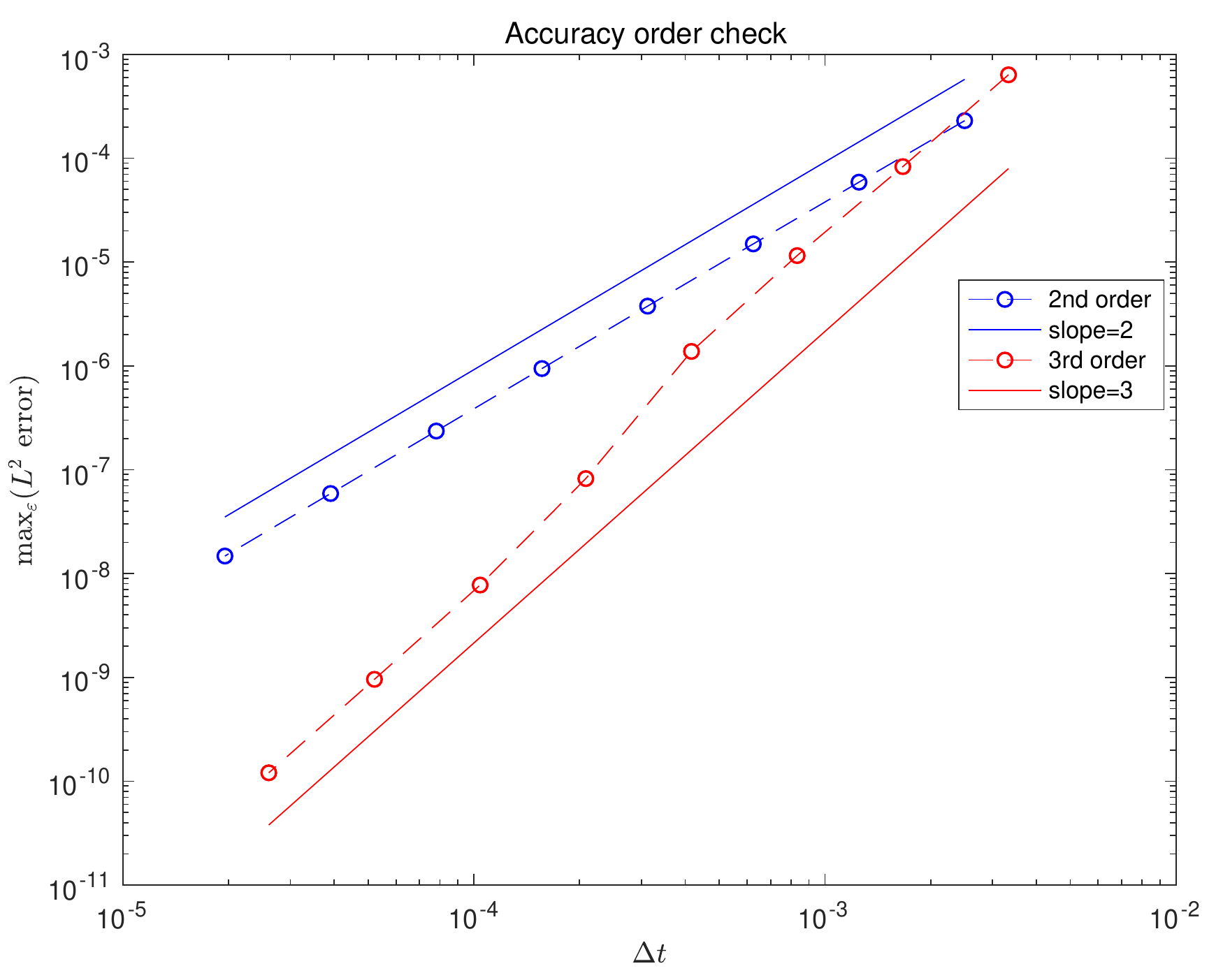}
	\caption{Nonlinear stiff hyperbolic relaxation system. $L^2$ error of the solutions computed by IMEX-BDF schemes. Top left and top right figures: second and third order IMEX-BDF schemes, respectively. In these two subfigures, horizontal axis is $\eps$ ranging from $1e-7$ to $1$, and different curves represent different values of $\Delta t$, given as $\Delta t=\Delta t_0,\Delta t_0/2,\Delta t_0/4,\dots$ from top curve to bottom curve, with $\Delta t_0 = 2.5\times 10^{-3}$ for the second order scheme and $\Delta t_0 = 3.3\times 10^{-3}$ for the third order scheme. Bottom figure is obtained as follows: for each scheme, take the maximal $L^2$ error among all values of $\eps$ for a fixed $\Delta t$.}
	\label{fig2}
\end{center}	
\end{figure}

\begin{remark}
The initial data we take here does not satisfy the conditions in Theorem~\ref{thm_accu}. Nevertheless the $q$-th order uniform accuracy can still be observed. This shows that the conditions of Theorem \ref{thm_accu} may not be optimal. However, if using \eqref{ini1} for the third order scheme, we do observe some order reduction. This means that certain high-order consistency of initial data is necessary to achieve uniform accuracy of high-order IMEX-BDF schemes.
\end{remark}

\subsection{The nonlinear stiff kinetic BGK equation}
\label{subsec:BGK}

We finally consider the kinetic BGK equation in one dimension
\begin{equation}
\partial_t f + v\partial_x f = \frac{1}{\eps}(M[f]-f).
\end{equation}
The spatial domain is taken as $x\in [0,2]$ with periodic boundary condition and discretized by the fifth order finite volume WENO scheme. The velocity domain is truncated into $[-|v|_{\text{max}},|v|_{\text{max}}]$ with $|v|_{\text{max}}=15$ and discretized by a finite difference scheme using $N_v=150$ grid points. The time step is chosen as $\Delta t = \frac{1}{3}\frac{\Delta x}{|v|_{\text{max}}}$ to satisfy the CFL condition.

We apply the second and third order IMEX-BDF schemes for time discretization. For the second order scheme, we take the initial data as
\begin{equation}
f(0,x,v) = M_{\rho,u,T},
\end{equation}
with
\begin{equation}\label{rhouT}
\rho(0,x) = 1+0.2\sin\pi x,\quad u(0,x) = 1,\quad T(0,x) = \frac{1}{1+0.2\sin\pi x} ,
\end{equation}
which is consistent up to $O(1)$. For the third order scheme, we take the initial data as
\begin{equation}
f(0,x,v) = M_{\rho,u,T}\left(1-\eps\left(\frac{(v-u)^2}{2T}-\frac{3}{2}\right)\frac{(v-u)\partial_x T}{T}\right),
\end{equation}
with \eqref{rhouT}, which is consistent up to $O(\eps)$. The starting values at $i\Delta t$, $i=0,\dots,q-1$, are prepared using an IMEX-RK scheme with a much smaller time step $\delta t = \Delta t/500$. We compute the solution to time $T=0.1$ and estimate the $L^2$ error of the solution $f_{\Delta t, \Delta x}$ as $\|f_{\Delta t, \Delta x}-f_{\Delta t/2,\Delta x/2}\|_{L^2_{x,v}}$.

The results are shown in Figure \ref{fig3}. The uniform $q$-th order accuracy can be observed, similar to the previous subsection, except for the third order scheme for which one sees a higher convergence rate because the error from the spatial discretization is always dominating.


\begin{figure}[htp]
\begin{center}
	\includegraphics[width=0.49\textwidth]{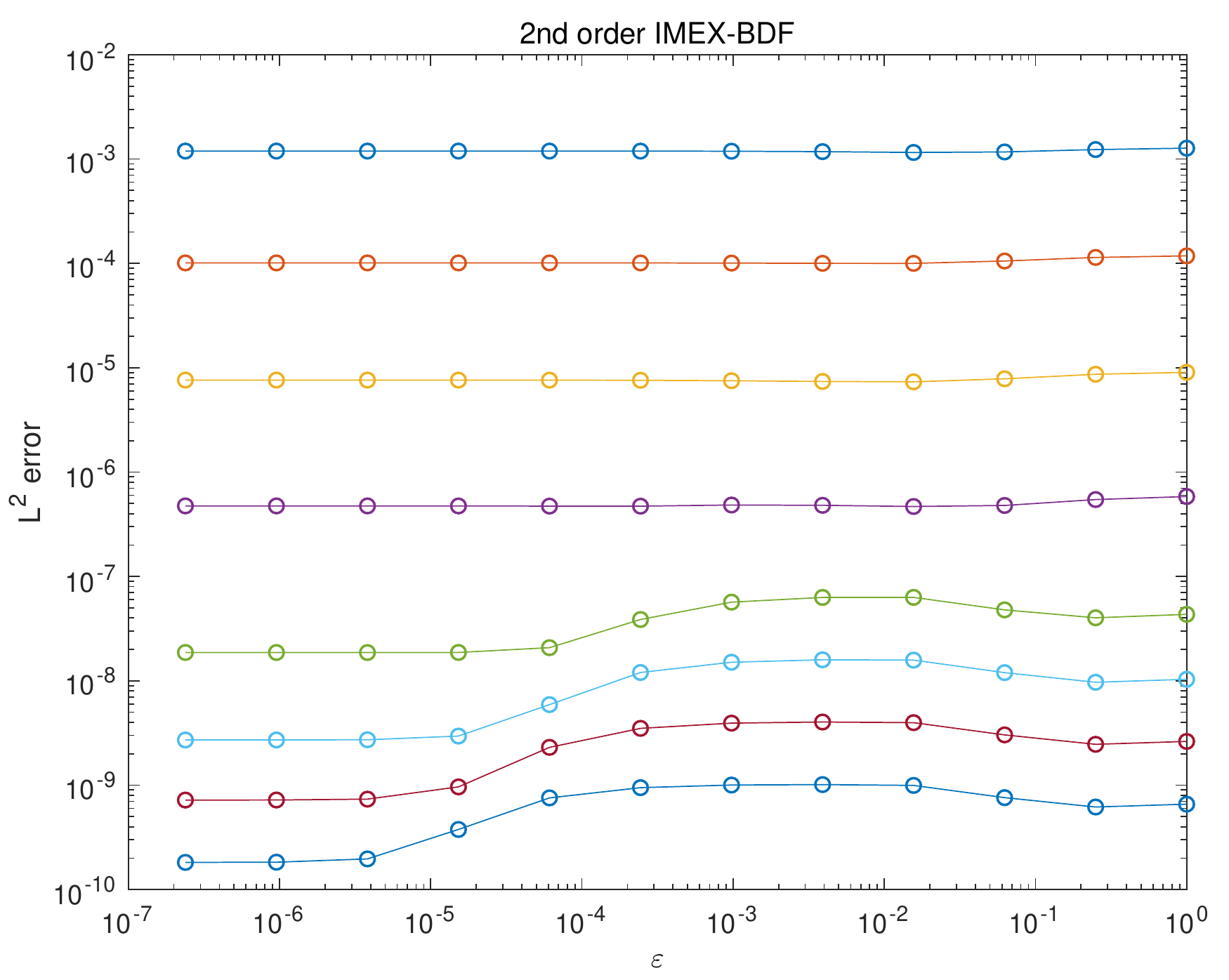}
	\includegraphics[width=0.49\textwidth]{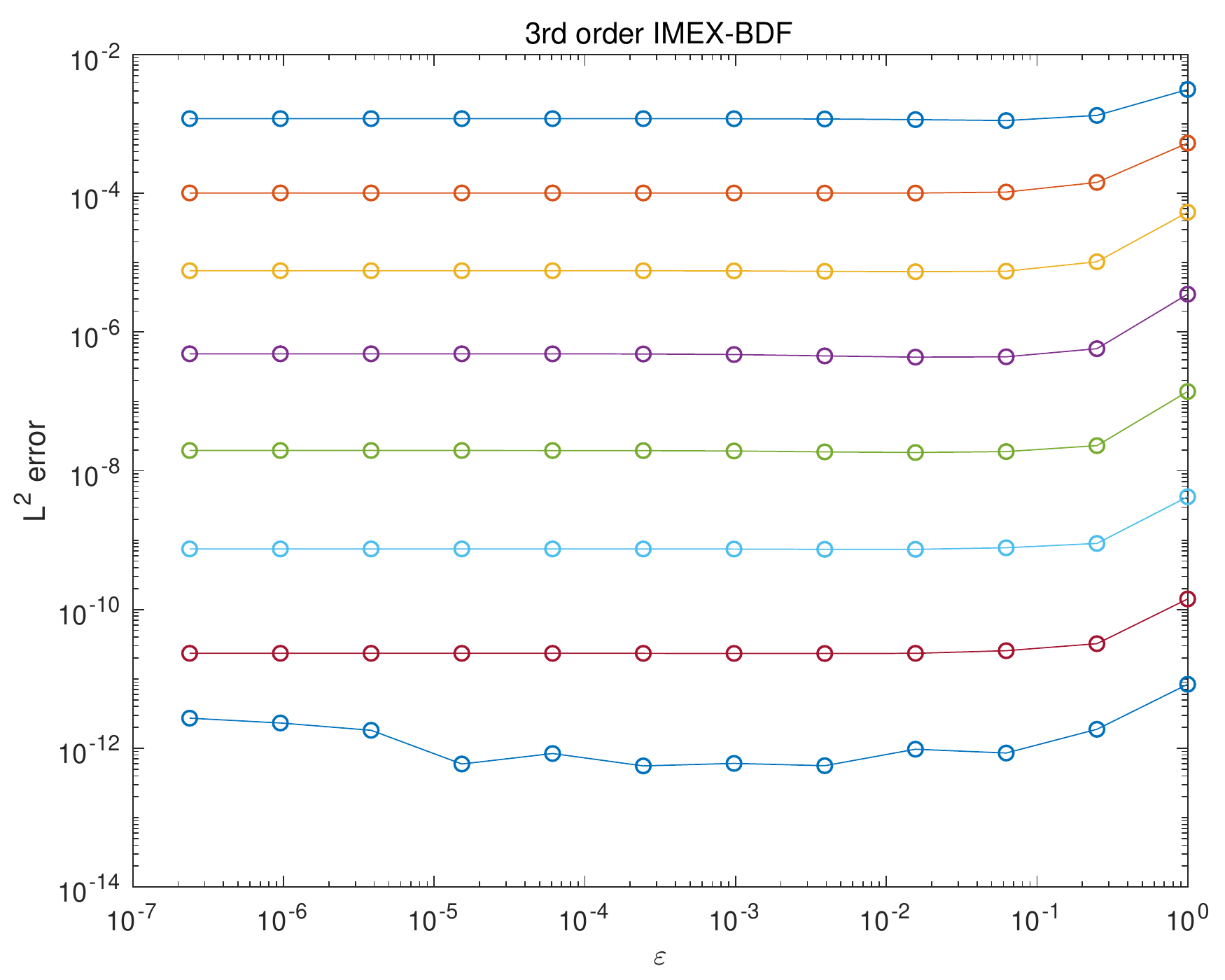}
	\includegraphics[width=0.49\textwidth]{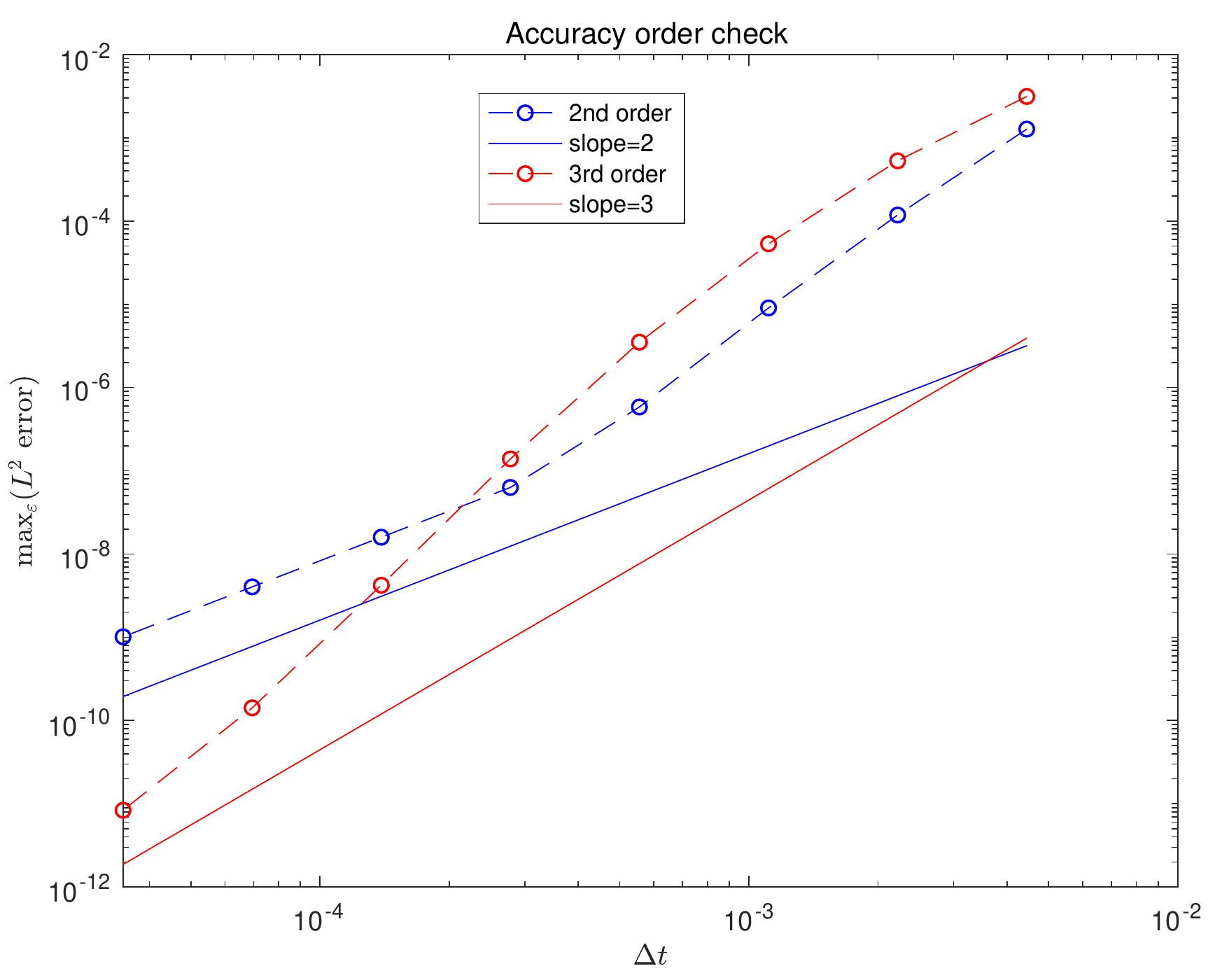}
	\caption{Nonlinear stiff kinetic BGK equation. $L^2$ error of the solutions computed by IMEX-BDF schemes. Top left and top right figures: second and third order IMEX-BDF schemes, respectively. In these two subfigures, horizontal axis is $\eps$ ranging from $1e-7$ to $1$, and different curves represent different values of $\Delta t$, given as $\Delta t=\Delta t_0,\Delta t_0/2,\Delta t_0/4,\dots$ from top curve to bottom curve, with $\Delta t_0 = 4.4\times 10^{-3}$. Bottom figure is obtained as follows: for each scheme, take the maximal $L^2$ error among all values of $\eps$ for a fixed $\Delta t$.}
	\label{fig3}
\end{center}	
\end{figure}



\section{Conclusion}
\label{sec:con}

The stiff kinetic equation (\ref{kinetic}) plays an important role in multiscale modeling by connecting mesoscopic kinetic and macroscopic fluid descriptions. Inspired by its structure, we study in this paper a simple linear hyperbolic system with stiff relaxation. Our main concern is to understand the accuracy of a class of IMEX methods, IMEX-BDF schemes, that are widely used to solve this kind of equations. By studying the regularity of the solution in time and introducing a new multiplier technique, we were able to establish uniform stability and accuracy of these schemes for (\ref{linearhyp}). Its extension to the variable coefficient case was also considered. We provided several numerical examples for both linear and nonlinear problems to validate our theoretical findings.

Regarding future work, one possibility is to consider high order spatial discretizations other than the spectral method, for example, the discontinuous Galerkin method. Another direction is to study the hyperbolic/kinetic equations in diffusive scaling. As opposed to the scaling in the current work which leads to a hyperbolic equation when $\eps \rightarrow 0$, the diffusive scaling leads to a diffusion equation in the limit. There is already some numerical analysis work in this case \cite{LM10, JLQX14}, whereas the time discretization is limited to first order. 



\section*{Acknowledgement}

JH is grateful to a discussion group held in Tianyuan Mathematical Center at Xiamen University, China in June 2019 for revisiting the order reduction phenomenon of IMEX-RK schemes, which motivates the current work.

\appendix

\section*{Appendix}

\section{IMEX-BDF schemes for the kinetic BGK equation}
\label{appendix:A}

In this appendix, we describe briefly the kinetic BGK equation (equation (\ref{kinetic}) with $\mathcal{Q}$ being the BGK operator \cite{BGK54}) along with its time discretization using IMEX-BDF schemes. Further details can be found in many references, e.g., \cite{DP17}.

The BGK equation reads
\begin{equation} \label{BGK}
\partial_t f+v\cdot \nabla_x f=\frac{1}{\varepsilon} (M[f]-f), \quad t>0, \quad x\in \Omega\subset \mathbb{R}^d, \quad v\in \mathbb{R}^d, \quad d=1,2,3,
\end{equation}
where $f=f(t,x,v)$ is the PDF, $\varepsilon$ is the Knudsen number, and $M$ is the Maxwellian, or local equilibrium, defined as
\begin{equation}
M[f]=\frac{\rho}{(2\pi T)^{\frac{d}{2}}}\exp\left(-\frac{|v-u|^2}{2T}\right),
\end{equation}
where $\rho$, $u$ and $T$ are density, bulk velocity and temperature given by the moments of $f$:
\begin{equation}
\rho = \int_{\mathbb{R}^{d}}f\,\rd{v}, \quad u=\frac{1}{\rho}\int_{\mathbb{R}^{d}}f v \,\rd{v}, \quad T=\frac{1}{d\rho}\int_{\mathbb{R}^{d}} f |v-u|^2\,\rd{v}.
\end{equation}
It is easy to verify the Maxwellian $M$ shares the same first $d+2$ moments as $f$:
\begin{equation} \label{prop}
\langle M[f] \phi \rangle=\langle f \phi \rangle, \quad \langle \cdot \rangle :=\int_{\mathbb{R}^{d}} \cdot \,\rd{v}, \quad \phi(v):=(1,v,|v|^2)^T, 
\end{equation}
and the moments can be represented using $\rho$, $u$ and $T$ as
\begin{equation}
\langle M[f] \phi \rangle=\langle f \phi \rangle=:U, \quad U=(\rho,\rho u,\rho u^2+d \rho T)^T.
\end{equation}

The IMEX-BDF scheme applied to (\ref{BGK}) reads
\begin{equation} \label{BGK-BDF}
\sum_{i=0}^q\alpha_i f^{n+i} + \Delta t \sum_{i=0}^{q-1}\gamma_i  v\cdot \nabla_x f^{n+i} = \frac{\beta\Delta t}{\eps}(M^{n+q}-f^{n+q}),
\end{equation}
where the coefficients $\alpha$, $\beta$ and $\gamma$ are given in Table~\ref{coeff_table}. 

The scheme (\ref{BGK-BDF}) appears nonlinearly implicit since $M^{n+q}$ depends on $f^{n+q}$. However, it can be implemented in an explicit manner using the property (\ref{prop}). Indeed, taking $\langle \cdot \phi\rangle$ on both sides of (\ref{BGK-BDF}) yields
\begin{equation}
\sum_{i=0}^q\alpha_i U^{n+i} + \Delta t \sum_{i=0}^{q-1}\gamma_i   \nabla_x \cdot \langle vf^{n+i}\phi\rangle = 0.
\end{equation}
Hence one can obtain $U^{n+q}$ first, which gives $\rho^{n+q}$, $u^{n+q}$, $T^{n+q}$ and consequently defines $M^{n+q}$. Then $f^{n+q}$ is explicitly solvable from (\ref{BGK-BDF}).

\section{Lemma \ref{lem_G} for $q=4$}
\label{appendix:B}

Here we list an approximate choice of coefficients for $q=4$. These values are obtained using symbolic computation such that \eqref{lem_G_11} and \eqref{lem_G_21} are satisfied up to an error of $10^{-31}$ for each coefficient of $u_iu_j$.
\begin{equation}\begin{split}
  &  g_{11} = 0.0039752881793877403062594960990749\\ &
    g_{22} = 0.064911795951738806179916997308306\\ &
    g_{33} = 0.15895411498724386738087416173813\\ &
    g_{44} = 0.094405276410813782029324113702474\\ &
    g_{12} = -0.015901152717550961225037984396299\\ &
    g_{13} = 0.023851729076326441837556976594449\\ &
    g_{14} = -0.015901152717550961225037984396299\\ &
    g_{23} = -0.099845362848676151804359071547141\\ &
    g_{24} = 0.068060968391835181509937875064459\\ &
    g_{34} = -0.11343769059020016642872820354501\\ &
    \eta_1 = 0.15803668922323725486664131361509\\ &
    \eta_2 = -0.71978015153831346379435821236894\\ &
    \eta_3 = 1.4954886593252805371567252971026\\ &
    d_1 = 0.90559472358918621797067588629753\\ &
    d_2 = 0.13803083956207431618956583677343\\ &
    a_{11} = 0.33641496341408589312936763149214\\ &
    a_{22} = 0.76333671636580335303312323666967\\ &
    a_{33} = 0.98143988982596163900489424649966\\ &
    a_{12} = -0.37897607563001842534574230771888\\ &
    a_{13} = 0.27087482555618781445745449601102\\ &
    a_{23} = -0.68412028602560052688774821729113\\ &
    c_1 = 0.58001289935145915953659169454951\\ &
    c_2 = -0.65339249532858690456475648019655\\ &
    c_3 = 0.46701517476433063541910705624076\\ &
    c_4 = -0.13623549527945483633692800818261\\
\end{split}\end{equation}
The smallest eigenvalues of $G$ and $A$ are approximately given by
\begin{equation}
\lambda_1(G)\approx 0.0000050314827031121361651672184255414,\quad \lambda_1(A)\approx 0.076485547272566738154682052748733.
\end{equation}

\bibliographystyle{plain}
\bibliography{hu_bibtex}

\end{document}